\documentclass[11pt, twoside]{amsart}

\usepackage{amsmath,amssymb,amsthm,enumitem,stmaryrd,mathrsfs}

\usepackage{thm-restate,hyperref}

\usepackage[pdftex]{graphicx}

\usepackage[top=1.5in, bottom=1.5in, left=1.5in, right=1.5in]{geometry}

\usepackage{tikz}

\allowdisplaybreaks

\theoremstyle{plain}
\newtheorem{theorem}{Theorem}[section]
\newtheorem{lemma}[theorem]{Lemma}

\newtheorem{prop}[theorem]{Proposition}

\newtheorem*{lemma*}{Lemma}
\newtheorem*{cor*}{Corollary}
\newtheorem*{theorem*}{Theorem}

\theoremstyle{definition}

\newtheorem{problem*}{Problem}
\newtheorem{example}{Example}
\newtheorem{definition}[theorem]{Definition}

\theoremstyle{remark}

\newtheorem*{fact*}{Fact}
\newtheorem*{remark}{Remark}

\newtheorem{question}{Question}

\let\oldproofname=\proofname
\renewcommand{\proofname}{\rm\bf{\oldproofname}}

\newcommand{\R}{\mathbb R}
\newcommand{\Q}{\mathbb Q}
\newcommand{\Z}{\mathbb Z}
\newcommand{\C}{\mathbb C}

\newcommand{\CP}{\mathbb C\text{P}}

\newcommand{\phii}{\varphi}
\newcommand{\ep}{\varepsilon}

\makeatletter
\providecommand*{\twoheadrightarrowfill@}{%
  \arrowfill@\relbar\relbar\twoheadrightarrow
}
\providecommand*{\twoheadleftarrowfill@}{%
  \arrowfill@\twoheadleftarrow\relbar\relbar
}
\providecommand*{\xtwoheadrightarrow}[2][]{%
  \ext@arrow 0579\twoheadrightarrowfill@{#1}{#2}%
}
\providecommand*{\xtwoheadleftarrow}[2][]{%
  \ext@arrow 5097\twoheadleftarrowfill@{#1}{#2}%
}
\makeatother

\newcommand{\op}[1]{\operatorname{{#1}}}

\newcommand{\mc}[1]{\mathcal{{#1}}}

\newcommand{\rank}{\op{rank}}

\newcommand{\Hom}{\op{Hom}}

\newcommand{\abs}[1]{\left\lvert#1\right\rvert}

\begin{document}
    \title{Chern Class Obstructions to Smooth Equivariant Rigidity}
    \author{Oliver H. Wang}
    \begin{abstract}
By work of Kirby-Siebenmann \cite{KirbySiebenmann} and Kervaire-Milnor \cite{KervaireMilnor}, there are only finitely many smooth manifolds homeomorphic to a given closed topological manifold.
A construction involving Whitehead torsion shows this is not the case equivariantly for smooth finite group actions on a product $M\times I$ (see \cite[p. 262-266]{BrowderHsiangProblem}). 
When $2$ has odd order in $\left(\Z/p\Z\right)^\times$, Schultz \cite{SchultzSpherelike} uses a different method involving the Atiyah-Singer index theorem and computations of Ewing \cite{EwingSpheresAsFPSets} to show that there are infinitely many equivariant smooth structures for certain actions of $G=\Z/p\Z$ on even dimensional spheres with fixed point set $S^2$.
These examples are constructed by finding infinitely many $G$-vector bundles over $S^2$ with vanishing Atiyah-Singer class and using these vector bundles to replace the normal bundle of $S^2\subseteq S^{2n}$.
We analyze when a manifold supports infinitely many $G$-vector bundles with vanishing Atiyah-Singer class and show that Schultz's examples of exotic equivariant manifolds can be extended to much greater generality.
As a consequence, we see that, for infinitely many primes $p$, there are infinitely many stable $G$-smoothings of a smooth $G$-manifold in the sense of Lashof \cite{LashofStableGSmoothing} whenever the fixed set has nonzero second rational cohomology.
\end{abstract}
	\maketitle

\tableofcontents

\section{Introduction}
For a closed topological manifold $X$, a smoothing of $X$ is defined to be a homeomorphism $Y\rightarrow X$ where $Y$ is a smooth manifold.
In \cite{KirbySiebenmann}, the set $TOP/O(X)$\footnote{We deviate from the classical notation of $\mc{S}^{TOP/DIFF}(X)$ so it will not be confused with a surgery theoretic structure set and to make the notation cleaner in the equivariant setting.} of smoothings up to isotopy is studied and it is shown that there is a bijection $TOP/O(X)\cong[X,TOP/O]$ where $TOP/O$ is an infinite loop space.
In particular, this set is a cohomology group so some computational methods are available.

The group $\op{Homeo}(X)$ acts on $TOP/O(X)$ and the quotient $\overline{TOP/O}(X)$ is the set of smooth manifolds homeomorphic to $X$ up to diffeomorphism.
This set is more difficult to compute.
When $n\ge5$, the group $\left[S^n,TOP/O\right]$ can be identified with the group of homotopy spheres of \cite{KervaireMilnor} and is therefore finite.
For $n<5$, Kirby-Siebenmann show that $[S^n,TOP/O]$ is finite by other means.
Consequently $TOP/O(X)$ and $\overline{TOP/O}(X)$ are finite sets.
We may define analagous sets in an equivariant setting.

\begin{definition}\label{def: equivariant smoothings}
Let $G$ be a finite group and let $X$ be a $G$-manifold.
A \emph{$G$-smoothing} of $X$ is an equivariant homeomorphism $\alpha:Y\rightarrow X$.
Two $G$-smoothings $\alpha_0$ and $\alpha_1$ are \emph{isotopic} if $\alpha_0$ is homotopic through $G$-homeomorphisms to $\alpha_0'$ where $\alpha_1^{-1}\circ\alpha_0'$ is a $G$-diffeomorphism.
Define $TOP/O_G(X)$ to be the set of isotopy classes of $G$-smoothings of $X$.
Define $\overline{TOP/O}_G(X)$ to be the equivariant diffeomorphism classes of smooth $G$-manifolds $Y$ which are equivariantly homeomorphic to $X$.
\end{definition}

\begin{remark}
If $X$ is closed and $\overline{TOP/O}_G(X)$ is infinite, then Kirby-Siebenmann's result implies there is a smooth structure on $X$ such that there are infinitely many periodic diffeomorphisms of $X$ which are conjugate in $\op{Homeo}(X)$ but not in $\op{Diff}(X)$.
\end{remark}

Schultz shows in \cite{SchultzSpherelike} that, for certain actions of $\Z/p\Z$ on $S^{2n}$ with fixed point set $S^2$, the set $\overline{TOP/O}_G\left(S^{2n}\right)$ is infinite contrary to the non-equivariant case.
Our goal is to generalize Schultz's construction to smooth $\Z/p\Z$-actions on manifolds whose fixed point set is not necessarily $S^2$.
Before explaining Schultz's work we introduce some notions fundamental to the study of smooth group actions.

For the remainder of the paper, let $G=\Z/p\Z$ where $p$ is an odd prime and let $g_0$ be a fixed generator.
If $M$ is a connected component of $X^G$, then the normal bundle of $M$ inherits the structure of a real $G$-vector bundle with fiber a real $G$-representation $V$.
We call $V$ the \emph{normal representation of $M$} and we say that $V$ is \emph{free} if $V^G=\{0\}$.
All representations obtained from normal bundles of fixed point sets are free.
A bundle of $G$-representations is said to be free if its fibers are free.

The nontrivial real irreducible $G$-representations are isomorphic to $\C$ where $g_0$ acts via multiplication by a primitive $p$-th root of unity.
Moreover, the representation determined by a primitive $p$-th root of unity $\zeta$ is isomorphic as a real representation to the one determined by the complex conjugate $\bar{\zeta}$.
So there are $\frac{p-1}{2}$-many nontrivial irreducible free real $G$-representations.
If $\nu$ is the normal bundle of $M$ as above, then $\nu$ decomposes as a sum $\nu\cong\bigoplus_{k=1}^{\frac{p-1}{2}}\nu_k$ of eigenbundles for $g_0$.
If $\phii:Y\rightarrow X$ is an equivariant diffeomorphism, $\phii^{-1}\left(M\right)\cong M$ and the normal bundle of $\phii^{-1}\left(M\right)$ is $\phii^*E$.

\subsection{Actions on Spheres}
In \cite{AtiyahSinger3}, Atiyah and Singer define the $G$-signature $\op{sign}_G(X)$ of a smooth $G$-manifold $X$.
This is an equivariant homotopy invariant of $X$ valued in the real representation ring of $G$.
For computational purposes, it is convenient to identify a representation with its character.
The Atiyah-Singer $G$-signature theorem states
\[
\op{sign}_G(X)(g)=\langle A(g,V)L\left(X^g\right)\mc{M}(g,\nu),\left[X^g\right]\rangle
\]
where $A(g,V)\in\C$, $L\left(X^g\right)$ is the $L$-genus and $\mc{M}(g,\nu)\in H^*\left(X^g;\C\right)$ is a (non-homogeneous) characteristic class of the normal bundle $\nu$ of $X^g$.
The class $\mc{M}(g,\nu)$ can be described in terms of the Chern classes of $\nu$.
As suggested by the notation, the number $A(g,V)$ depends only on the element $g\in G$ and the representation $V$.
We call $A(g_0,V)\mc{M}(g_0,\nu)$ the Atiyah-Singer class.
In Lemma \ref{lem: AS determined by generator} we show that the representation is determined by the Atiyah-Singer class.

Ewing \cite{EwingSemifree} applies the Atiyah-Singer $G$-signature theorem to determine the possible Chern classes of the normal bundle of the fixed point set of a smooth, semifree action of a cyclic group.
When $X$ is a sphere, the $G$-signature always vanishes since spheres have no middle-dimensional cohomology.
Moreover, the fixed point set is a rational homology $2n$-sphere by Smith theory  so $\mc{L}\left(X^{g_0}\right)=1$ and
\[
\mc{M}(g_0,\nu)=1+\sum_{k=1}^{\frac{p-1}{2}}\Phi_{n,k} c_n\left(\nu_k\right)
\]
where the $\Phi_{n,k}$ are elements of $\Q\left(\zeta\right)$.
Ewing shows that, unless $n=1$ and $2$ has odd order in $\left(\Z/p\Z\right)^\times$, the elements $\Phi_{n,k}$ are $\Q$-linearly independent as $k$ varies.
Since the Atiyah-Singer class must vanish, this implies that, outside the special case, the Chern classes of the normal bundle must vanish.

In the case $p$ has odd order in $\left(\Z/p\Z\right)^\times$, Ewing shows that the set $\left\{\Phi_{1,k}\right\}_{k=1}^{\frac{p-1}{2}}$ is $\Q$-linearly dependent.
If $V$ is the normal representation of $S^2$ then this implies that, provided $V$ contains enough nonzero eigenbundles, there are infinitely many $G$-vector bundles $E$ over $S^2$ with fiber $V$ whose Atiyah-Singer class vanishes.
Schultz shows in \cite{SchultzSpherelike} that infinitely many of these $G$-vector bundles can be realized as normal bundles of $\Z/p\Z$-actions on homotopy $2n$-spheres.

The above results motivate the following questions.

\begin{question}\label{question: AS vanishing}
Suppose $E$ is a free $G$-vector bundle over a CW-complex $M$.
When does the Atiyah-Singer class vanish?
\end{question}

\begin{question}\label{question: infinitely many vb}
Given a free $G$-representation $V$, when are there infinitely many $G$-vector bundles over $M$ with fiber $V$ and vanishing Atiyah-Singer class?
\end{question}

\begin{question}\label{question: realizing exotic actions}
Suppose $G$ acts on $X$ and let $M$ be a component of $X^G$ with trivial normal bundle $M\times V$.
If there are infinitely many $G$-vector bundles over $M$ with fiber $V$ and vanishing Atiyah-Singer class, are there infinitely many $G$-smoothings of $X$ realizing these bundles as normal bundles?
\end{question}

\subsection{Main Results}
If $E$ is a free $G$-vector bundle, we write $E=\bigoplus_{k=1}^{\frac{p-1}{2}}E_k$ for the decomposition into eigenbundles where $g_0$ acts on the fiber of $E_k$ via multiplication by $\zeta^k$.
In this case, we give a complete answer to Question \ref{question: AS vanishing}.

\begin{restatable}{theorem}{ChernClass}\label{thm: Chern class restriction}
Let $E=\bigoplus_{k=1}^{\frac{p-1}{2}}E_k$ be a free $G$-vector bundle over a space $M$.
There is an equality $A\left(g_0,V\right)\mc{M}\left(g_0,E\right)=1$ if and only if both of the following hold.
\begin{enumerate}
\item $\sum_{k=1}^{\frac{p-1}{2}}c_1\left(E_k\right)\Phi_{1,k}=0\in H^2(M;\C)$;
\item For each $k$ and $n\ge1$, $c_n\left(E_k\right)=\frac{1}{n!}c_1\left(E_k\right)^n$.
\end{enumerate}
\end{restatable}

When $2$ has even order in $\left(\Z/p\Z\right)^\times$, Ewing shows that the first condition above implies $c_1\left(E_k\right)=0$ for all $k$.
Otherwise, the $\Q$-span of $\left\{\Phi_{1,k}\right\}_{k=1}^{\frac{p-1}{2}}$ is a $\frac{(p-1)(t-1)}{2t}$-dimensional vector space where $t$ is the order of $2$ in $\left(\Z/p\Z\right)^\times$.
For simplicity, let $u:=\frac{(p-1)(t-1)}{2t}$.

\begin{definition}\label{def: sufficient nilpotence}
Let $G=\Z/p\Z$ where $p$ is such that $2$ has odd order in $\left(\Z/p\Z\right)^{\times}$.
Suppose $M$ is a space and $V$ is a free $G$-representation.
An element $\beta\in H^2\left(M;\Z\right)$ is \emph{sufficiently nilpotent with respect to $V$} if there is an $N>0$ such that the following hold.
\begin{enumerate}
\item $\beta^{N+1}=0$,
\item $V$ contains $(u+1)$-many irreducible real representations with multiplicity at least $N$.
\end{enumerate}
\end{definition}

\begin{example}
If $V$ is contains one copy of each nontrivial irreducible real representation, then every $\beta$ satisfying $\beta^2=0$ is sufficiently nilpotent with respect to $V$.
\end{example}

We can now state a partial answer to Question \ref{question: infinitely many vb}.

\begin{restatable}{theorem}{VBExistenceThm}\label{thm: VBExistence}
Let $G=\Z/p\Z$ and let $V$ be a free $G$-representation.
Suppose $M$ is homotopy equivalent to a finite CW-complex.

\begin{enumerate}
\item If $2$ has even order in $\left(\Z/p\Z\right)^\times$ or if $H^2(M;\Q)=0$ then there are only finitely many $G$-vector bundles over $M$ with fiber $V$ and vanishing Atiyah-Singer class.

\item If $2$ has odd order in $\left(\Z/p\Z\right)^\times$ and there is a nonzero $\beta\in H^2\left(M;\Q\right)$ sufficiently nilpotent with respect to $V$, then there are infinitely many $G$-vector bundles over $M$ with fiber $V$ and vanishing Atiyah-Singer class.
\end{enumerate}
\end{restatable}

To more easily state our answer to Question \ref{question: realizing exotic actions} we introduce another auxiliary definition.

\begin{definition}\label{def: exotic normal bundle}
Suppose $G$ acts smoothly on a manifold $X$ and let $M$ be a component of the fixed point set.
A $G$-vector bundle $E$ over $M$ is an \emph{exotic normal bundle of $(X,M)$} if $G$ acts smoothly on a manifold $Y$ and there is an equivariant homeomorphism $f:Y\rightarrow X$ such that $f^{-1}(M)$ has normal bundle $E$.
\end{definition}

\begin{restatable}{theorem}{FPRealization}\label{thm: FPRealization}
Suppose $G=\Z/p\Z$ acts smoothly on a manifold $X$.
Let $M$ be a component of $X^G$ whose normal bundle is $M\times V$ with $V$ a free $G$-representation.
Suppose $M$ is homotopy equivalent to a finite CW-complex and admits infinitely many $G$-vector bundles with fiber $V$ and vanishing Atiyah-Singer class.
Then,
\begin{enumerate}
\item Infinitely many of these vector bundles may be realized as exotic normal bundles of $(X,M)$,
\item The first Chern classes of these exotic normal bundles occupy infinitely many $GL_{\dim_{\Q}H^2(M;\Q)}(\Z)$-orbits of $H^2(M;\Q)$ so $\overline{TOP/O}_G(X)$ is infinite.
\end{enumerate}
\end{restatable}




\begin{remark}
There are infinitely many primes $p$ where $2$ has odd order in $(\Z/p\Z)^\times$.
Indeed, this is true whenever $p\equiv 7$ modulo $8$ and it occurs infinitely many times when $p\equiv 1$ modulo $8$ (see \cite{EwingSemifree}).
Outside these cases, $2$ always has even order in $(\Z/p\Z)^\times$.
\end{remark}

\begin{example}
Suppose $V$ is a $2n$-dimensional real representation of $G$ which contains at least one copy of each nontrivial irreducible representation.
Let $DV$ denote the unit disk of $V$ and let $SV$ denote the unit sphere.
Then, $S^{2n+2}=\left(S^2\times DV\right)\cup_{S^2\times SV}\left(D^2\times S^V\right)$ has a $G$-action with fixed point set $S^2$.
The normal representation is $V$ and the generator of $H^2\left(S^2;\Q\right)$ is sufficiently nilpotent with respect to $V$ so  $\overline{TOP/O}_G\left(S^{2n}\right)$ is infinite.
This is the example of Schultz.
\end{example}

\begin{example}
If $M$ is smooth then the normal bundle of $M$ diagonally embedded in $M^{\times p}$ is $\tau_M^{\oplus p-1}$ where $\tau_M$ is the tangent bundle.
If $M^{\times p}$ is given a $G$-action via cyclically permuting coordinates, then the normal bundle of $M$ is $\tau_M\otimes\R[G]/\R$.
In particular, if $M$ is also closed and parallelizable such that $H^2\left(M;\Q\right)$ is nonzero and if $p$ is a prime such that $2$ has odd order in $\left(\Z/p\Z\right)^\times$ then Theorem \ref{thm: FPRealization} implies $\overline{TOP/O}_G\left(M^{\times p}\right)$ is infinite.

Let $V$ be the reduced regular representation and let $S^V$ denote the one point compactification of $V$.
If $M$ is only stably parallelizable with $H^2\left(M;\Q\right)\neq0$ and $p$ is as above, then $\overline{TOP/O}_G\left(M^{\times p}\times S^V\right)$ is infinite.
\end{example}

\begin{example}
For a more complicated example, let $M$ and $V$ be as in Theorem \ref{thm: FPRealization}.
Suppose $M$ is closed.
Then $M\times SV$ is a closed manifold with a free $G$-action which nonequivariantly bounds.
By equivariant cobordism theory \cite{ConnerFloyd} a disjoint union of $M\times SV$ bounds a smooth, compact manifold $X'$ with free $G$-action.
Define $X$ to be the manifold obtained by gluing copies of $M\times V$ to each boundary component of $X'$.
Then $X$ will satisfy the hypotheses of Theorem \ref{thm: FPRealization}.

Modifying this construction also shows that there are infinitely many $G$-manifolds $X$ for which the hypotheses of \ref{thm: FPRealization} do not hold.
\end{example}

The hypotheses on the normal bundle of the fixed point set can be removed when we stabilize as in \cite{LashofStableGSmoothing}.
Lashof defines a stable $G$-smoothing of a $G$-manifold $X$ is a $G$-smoothing of $X\times\rho$ for a finite dimensional $G$-representation $\rho$.
Two stable $G$-smoothings $\alpha_i:Y_i\rightarrow X\times\rho_i$, $i=0,1$ are stably isotopic if there are representations $\sigma_0$ and $\sigma_1$ such that $\alpha_i\times\sigma_i:Y_i\times\sigma_i\rightarrow X\times\rho_i\times\sigma_i$ are isotopic.
Let $TOP/O_G^{st}(X)$ denote the set of stable isotopy classes of stable $G$-smoothings of $X$.

\begin{restatable}{theorem}{StabilityThm}\label{thm: stability}
Let $G=\Z/p\Z$ where $p$ is such that $2$ has odd order in $\left(\Z/p\Z\right)^\times$.
Let $X$ be a smooth $G$-manifold.
If $H^2\left(X^G;\Q\right)$ is nonzero for some component $M$ of $X^G$ homotopy equivalent to a finite CW-complex, then $TOP/O_G^{st}(X)$ is infinite.
In particular, if $X$ is closed and $H^2\left(X^G;\Q\right)\neq0$ then $TOP/O_G^{st}(X)$ is infinite.
\end{restatable}

\subsection{Outline}
The proof of Theorem \ref{thm: Chern class restriction} has a large computational component and will be the subject of Section \ref{section: exp vb}.
In \cite{SchultzSpherelike}, Schultz exploits the fact that homotopy classes of maps from spheres into various classifying spaces have abelian group structures.
We do not have this luxury at our level of generality.
It turns out that the maps we are concerned with will factor through $\CP^N$ for a sufficiently large $N$ and self-maps of $\CP^N$ serve as a replacement.
We elaborate on this and prove Theorem \ref{thm: VBExistence} in Section \ref{section: construction of vb}.

The idea of the proof of Theorem \ref{thm: FPRealization} is as follows: if $M\subseteq X^G$ has normal bundle $M\times V$, we remove $M\times V$ and glue in $E$ where $E$ is the total space of some free $G$-vector bundle on $M$.
To do this, we need $M\times SV$ to be equivariantly diffeomorphic to the unit sphere bundle $SE$.
We introduce block bundles in Section \ref{section: block bundles} and apply result of Cappell-Weinberger to show that, if $E$ has vanishing Atiyah-Singer class, then $SE/G$ is almost equivalent to $M\times SV/G$ as lens space block bundles over $M$.
In Section \ref{section: smoothing}, we show that this equivalence can be taken to be a diffeomorphism and we prove Theorem \ref{thm: FPRealization}.
In Section \ref{section: nontrivial normal bundles} we give some remarks on the necessity of the trivial normal bundle hypothesis in our theorems and we prove Theorem \ref{thm: nontrivial normal bundles} which is a more general version of Theorem \ref{thm: FPRealization}.
This theorem is used to prove Theorem \ref{thm: stability}.

\subsection{Acknowledgments}
The author would like to thank Shmuel Weinberger for many helpful conversations, especially regarding the paper \cite{CappellWeinbergerSimpleAS}.

\section{Exponential Vector Bundles}\label{section: exp vb}

In this section, we determine necessary and sufficient conditions for the vanishing of the Atiyah-Singer class.
The main result of this section is Theorem \ref{thm: Chern class restriction}.

\subsection{The Atiyah-Singer Classes}\label{subsection: AS classes}

For the convenience of the reader and to establish notation, we review Hirzebruch's theory of multiplicative sequences and its application in the Atiyah-Singer $G$-signature formula.
Details can be found in \cite[Chapter 1]{Hirzebruch} and \cite[Section 6]{AtiyahSinger3}
Recall that, if $E\rightarrow X$ is a complex rank $n$ vector bundle, then the splitting principle asserts there is a space $P(E)$ with a map $f:P(E)\rightarrow X$ where $f^*E$ splits into a sum of line bundles $L_1\oplus\cdots\oplus L_n$ and $f^*:H^*(X)\rightarrow H^*(P(E))$ is injective.
So in $H^*(P(E))$, the total Chern class of $E$ factors as
\[
c_*(E)=1+c_1(E)+\cdots+c_n(E)=\prod_{j=1}^n(1+c_1(L_j)).
\]
This motivates the use of formal factorizations used below.

Fix a commutative ring $R$ and let $R[c_1,c_2,\cdots]$ be the graded commutative ring of polynomials in $c_j$ where $c_j$ has grading $2j$ (we deviate slightly from the notation of \cite{Hirzebruch} here).
Similarly, we consider $R[c_1,\cdots,c_j]$ as a graded commutative ring and, for convenience, we set $c_0=1\in R$.
A \emph{multiplicative sequence} $\{K_j\}$ is a sequence of polynomials where $K_j\in R[c_1,\cdots,c_j]$ is homogeneous of degree $j$, where $K_0=1$ and such that, if there is a formal factorization
\[
1+c_1z+c_2z^2+\cdots=\left(1+c_1'z+c_2'z^2+\cdots\right)\left(1+c_1''z+c_2''z^2+\cdots\right),
\]
then
\[
\sum_{j=0}^{\infty}K_j(c_1,\cdots,c_j)z^j=\sum_{j=0}^\infty K_j(c_1',\cdots,c_j'')z^j\sum_{k=0}^\infty K_k(c_1'',\cdots, c_k'')z^k.
\]

Suppose $Q(z)=\sum_{j=0}^\infty b_j z^j$ is a formal power series with coefficients in $R$.
If $b_0=1$, then we can assign a multiplicative sequence $\{K_j\}$ as follows.
To determine $K_j$, let $m\ge j$ and suppose there is a formal factorization
\[
1+b_1z+\cdots+b_mz^m=\prod_{k=1}^m(1+\beta_kz)
\]
where each $\beta_k$ is of degree $1$.
Suppose $j_1\ge j_2\ge j_3\ge\cdots\ge j_r$ and that $j_1+\cdots+j_r=j$.
Then, the coefficient of $c_{j_1}c_{j_2}\cdots c_{j_r}$ in $K_j(c_1,\cdots,c_j)$ is the sum of \emph{distinct} $S_j$-translates of $\beta_{j_1}\cdots\beta_{j_r}$.
As an example, the coefficient of $c_j$ in $K_j$ is $\beta_1^j+\beta_2^j+\cdots+\beta_m^j$ and the coefficient of $c_1^j$ is the $j$-th elementary symmetric polynomial on $\beta_1,\cdots,\beta_m$.
So long as $m\ge j$, these coefficients are well-defined.
We will let $\tau(j_1,\cdots,j_r)$ denote the coefficient of $c_{j_1}\cdots c_{j_r}$.

Consider a free $\Z/p\Z$-vector bundle $E$ over a manifold $M$.
This breaks into a sum of eigenbundles $E=\bigoplus_{k=1}^{\frac{p-1}{2}}E_k$ where a given generator $g\in\Z/p\Z$ acts by a primitive $p$-th root of unity $\zeta^k$ on $E_k$ and such that $\zeta^k\neq \zeta^{k'},\bar{\zeta}^{k'}$ for $k\neq k'$.

Let $\left\{\mc{M}_r^{\zeta^k}\left(c_1,\cdots,c_r\right)\right\}$ be the multiplicative sequence determined by the power series associated to
\[
\left(\frac{\zeta^k-1}{\zeta^k+1}\right)\left(\frac{\zeta^k e^z+1}{\zeta^k e^z-1}\right).
\]
Define
\[
\mc{M}^{\zeta^k}\left(E_k\right):=\sum_{r=0}^\infty\mc{M}_r^{\zeta^k}\left(c_1\left(E_k\right),\cdots,c_r\left(E_k\right)\right)
\]
where $c_1\left(E_k\right),\cdots,c_r\left(E_k\right)\in H^*\left(M;\C\right)$ are Chern classes of the vector bundle $E_k$.
The complex number showing up in the Atiyah-Singer index theorem is
\[
A(g,V)=\prod_{k=1}^{\frac{p-1}{2}}\left(\frac{\zeta^k+1}{\zeta^k-1}\right)^{\rank_{\C}\left(E_k\right)}
\]
and the class $\mc{M}(g,E)$ is
\[
\mc{M}(g,E):=\prod_{k=1}^{\frac{p-1}{2}}\mc{M}^{\zeta^k}\left(E_k\right).
\]

Choose an integer $m$ such that $m>\rank_{\C}\left(E_k\right)$ for all $k$ and consider a formal factorization $\prod_{j=1}^m(1+\beta_{j,k}z)$ of the first $m$ terms of the power series $\left(\frac{\zeta^k-1}{\zeta^k+1}\right)\left(\frac{\zeta^k e^z+1}{\zeta^k e^z-1}\right)$.
We may write $\mc{M}^{\zeta^k}\left(E_k\right)$ as
\[
\mc{M}^{\zeta^k}(E_k)=\sum_{r=0}^m\sum_{\substack{j_1\ge\cdots\ge j_\ell>0\\j_1+\cdots+j_\ell=r}}\tau\left(j_1,\cdots,j_\ell\right)\left(\zeta^k\right)c_{j_1}(E_k)\cdots c_{j_\ell}\left(E_k\right).
\]

\begin{remark}
The $\Phi_{n,k}$ in the introduction and in Theorem \ref{thm: Chern class restriction} are the numbers $\tau(n)\left(\zeta^k\right)$.
\end{remark}

We will rely on results from \cite{EwingSemifree}, summarized below, for our analysis of the Atiyah-Singer class.

\begin{lemma}\label{lem: Ewing}
If $r>1$ then $\left\{\tau(r)\left(\zeta\right),\tau(r)\left(\zeta^2\right),\cdots,\tau(r)\left(\zeta^{\frac{p-1}{2}}\right)\right\}$ is a $\Q$-linearly independent set.
If $r=1$, then this set is $\Q$-linearly independent if and only if $2$ has even order in $\left(\Z/p\Z\right)^\times$.
Moreover, when $2$ has odd order in $\left(\Z/p\Z\right)^{\times}$, the span of this set has dimension $\frac{(p-1)(t-1)}{2t}$ where $t$ is the order of $2$ in $\left(\Z/p\Z\right)^{\times}$.
\end{lemma}

The following observation will be important later.
\begin{lemma}\label{lem: Galois invariance of as}
The numbers $\tau\left(j_1,\cdots,j_{\ell}\right)$ are in $\Q(\zeta)$.
Moreover, if $\sigma\in Gal(\Q(\zeta)/\Q)$ then $\sigma\left(\tau\left(j_1,\cdots,j_{\ell}\right)\left(\zeta^k\right)\right)=\tau\left(j_1,\cdots,j_{\ell}\right)\left(\sigma\cdot\zeta^k\right)$.
\end{lemma}
\begin{proof}
Let $\sigma$ be the field automorphism determined by $\zeta\mapsto\zeta^n$.
Suppose
\[
\left(\frac{\zeta^k-1}{\zeta^k+1}\right)\left(\frac{\zeta^k e^z+1}{\zeta^k e^z-1}\right)=1+b_{1,k}z+b_{2,k}z^2+b_{3,k}z^3+\cdots
\]
is a power series expansion.
Each $b_{j,k}\in\Q(\zeta)$ and $\sigma\cdot b_{j,k}=b_{j,nk}$, i.e. $\sigma$ sends the power series for 
$\left(\frac{\zeta^k-1}{\zeta^k+1}\right)\left(\frac{\zeta^k e^z+1}{\zeta^k e^z-1}\right)$ to the power series for 
$\left(\frac{\zeta^{nk}-1}{\zeta^{nk}+1}\right)\left(\frac{\zeta^{nk} e^z+1}{\zeta^{nk} e^z-1}\right)$.

Let $e_{\ell,k}$ denote the $\ell$-th elementary symmetric polynomial of the $\beta_{j,k}$
The factorization $\prod_{j=1}^m(1+\beta_{j,k}z)$ of the first $m$ terms implies that $e_{\ell,k}=b_{\ell,k}$ for $\ell\le m$ so $\sigma\cdot e_{\ell,k}=e_{\ell,nk}$.
Since $\tau\left(j_1,\cdots,j_{\ell}\right)\left(\zeta^k\right)$ is an algebraic combination of the $e_{\ell,k}$, we see that it is indeed in $\Q(\zeta)$ and that
\[
\sigma\left(\tau\left(j_1,\cdots,j_{\ell}\right)\left(\zeta^k\right)\right)=\tau\left(j_1,\cdots,j_{\ell}\right)\left(\zeta^{nk}\right)
\]
as desired.
\end{proof}

\ChernClass*

\begin{proof}
Suppose $\prod_{k=1}^{\frac{p-1}{2}}\mc{M}^{\zeta^k}\left(E_k\right)=1$.
It is clear that the first condition must hold since the sum is the part of $\prod_{k=1}^{\frac{p-1}{2}}\mc{M}^{\zeta^k}\left(E_k\right)$ in cohomological degree $2$.
To prove that $c_n\left(E_k\right)=\frac{1}{n!}c_1\left(E_k\right)^n$, we use induction.
The case $n=1$ is vacuous.

Suppose that $n\ge2$ and that $c_j\left(E_k\right)=\frac{1}{j!}c_1\left(E_k\right)^j$ for all $k$ and all $j\le n-1$.
In degree $2n$, the product $\prod_{k=1}^{\frac{p-1}{2}}\mc{M}^{\zeta^k}(E_k)$ can be expressed as
\[
\left(\prod_{k=1}^{\frac{p-1}{2}}\mc{M}^{\zeta^k}\left(E_k\right)\right)_{2n}=\sum_{\substack{\ell_1,\cdots,\ell_{\frac{p-1}{2}}\ge0\\\ell_1+\cdots+\ell_{\frac{p-1}{2}}=n}}\prod_{k=1}^{\frac{p-1}{2}}\sum_{\substack{j_1\ge\cdots\ge j_r>0\\j_1+\cdots+j_r=\ell_k}}\tau(j_1,\cdots,j_r)(\zeta^k)c_{j_1}(E_k)\cdots c_{j_r}(E_k)
\]
where the inner sum is taken to be $1$ if $\ell_k=0$ and the subscript on the left hand side indicates that we are restricting to the cohomological degree $2n$ part.
It follows from the definition of $\tau(1)\left(\zeta^k\right)$ that
\[
\tau(1)\left(\zeta^k\right)^{\ell_k}c_1\left(E_k\right)^{\ell_k}=\sum_{\substack{j_1,\cdots,j_n\ge0\\j_1+\cdots+j_n=\ell_k}}\frac{\ell_k!}{j_1!\cdots j_n!}\beta_{1,k}^{j_1}\cdots\beta_{n,k}^{j_n}c_1\left(E_k\right)^{\ell_k}.
\]
If $\ell_k<n$, then the inductive hypothesis and the definition of $\tau\left(j_1,\cdots,j_r\right)$ gives
\begin{align*}
\tau(1)\left(\zeta^k\right)^{\ell_k}c_1(E_k)^{\ell_k}&=\sum_{\substack{j_1,\cdots,j_n\ge0\\j_1+\cdots+j_n=\ell_k}}\ell_k!\beta_{1,k}^{j_1}\cdots\beta_{n,k}^{j_n}c_{j_1}\left(E_k\right)\cdots c_{j_n}\left(E_k\right)\\
&=\sum_{\substack{j_1\ge\cdots\ge j_r>0\\j_1+\cdots+j_r=\ell_k}}\ell_k!\tau(j_1,\cdots,j_r)c_{j_1}\left(E_k\right)\cdots c_{j_r}\left(E_k\right).
\end{align*}
From this, we conclude
\begin{equation}\label{eq: sum sub 1}
\sum_{\substack{j_1\ge\cdots\ge j_r>0\\j_1+\cdots+j_r=\ell_k}}\tau\left(j_1,\cdots,j_r\right)c_{j_1}\left(E_k\right)\cdots c_{j_r}\left(E_k\right)=\frac{1}{\ell_k!}\tau(1)\left(\zeta^k\right)^{\ell_k}c_1\left(E_k\right)^\ell_k.
\end{equation}
Similarly, if $\ell_k=n$, we get
\begin{align*}
\tau(1)\left(\zeta^k\right)^nc_1\left(E_k\right)^n&=n!\tau(n)\left(\zeta^k\right)c_1\left(E_k\right)^n+\sum_{\substack{n>j_1,\cdots,j_n\ge0\\j_1+\cdots+j_n=n}}n!\beta_{1,k}^{j_1}\cdots\beta_{n,k}^{j_n}c_{j_1}(E_k)\cdots c_{j_n}\left(E_k\right)\\
&=n!\tau(n)\left(\zeta^k\right)c_1\left(E_k\right)^n+\sum_{\substack{n>j_1\ge\cdots\ge j_r>0\\j_1+\cdots+j_r=n}}n!\tau\left(j_1,\cdots,j_r\right)c_{j_1}\left(E_k\right)\cdots c_{j_r}\left(E_k\right).
\end{align*}
This implies
\begin{align}\label{eq: sum sub 2}
\begin{split}
\sum_{\substack{n\ge j_1\ge \cdots\ge j_r>0\\j_1+\cdots+j_r=n}}\tau\left(j_1,\cdots,j_r\right)c_{j_1}\left(E_k\right)\cdots c_{j_r}\left(E_k\right)=&\frac{1}{n!}\tau(1)\left(\zeta^k\right)^nc_1\left(E_k\right)^n
\\&-\tau(n)\left(\zeta^k\right)c_1\left(E_k\right)^n+\tau(n)\left(\zeta^k\right)c_n\left(E_k\right).
\end{split}
\end{align}
Using Equations (\ref{eq: sum sub 1}) and (\ref{eq: sum sub 2}) above, we may rewrite $\left(\prod_{k=1}^{\frac{p-1}{2}}\mc{M}^{\zeta^k}\left(E_k\right)\right)_{2n}$ as follows.
\begin{align}
\left(\prod_{k=1}^{\frac{p-1}{2}}\mc{M}^{\zeta^k}\left(E_k\right)\right)_{2n}&=\sum_{k=1}^{\frac{p-1}{2}}\left(\tau(n)\left(\zeta^k\right)c_n\left(E_k\right)-\frac{1}{n!}\tau(n)\left(\zeta^k\right)c_1\left(E_k\right)^n\right)\label{eq: unsimplified AS class}\\
&+\sum_{\substack{\ell_1,\cdots,\ell_n\ge0\\\ell_1+\cdots+\ell_{\frac{p-1}{2}}=n}}\prod_{k=1}^{\frac{p-1}{2}}\frac{1}{\ell_k!}\tau(1)\left(\zeta^k\right)^{\ell_k}c_1\left(E_k\right)^{\ell_k}\nonumber
\end{align}

Since $\sum_{k=1}^{\frac{p-1}{2}}\tau(1)\left(\zeta^k\right)c_1\left(E_k\right)=0$,
\begin{equation}\label{eq: big vanishing class}
\sum_{k_1,\cdots,k_n\in\left\{1,\cdots,\frac{p-1}{2}\right\}}\prod_{m=1}^n\tau(1)\left(\zeta^{k_m}\right)c_1\left(E_m\right)=0.
\end{equation}
Suppose $\left(k_1,\cdots,k_n\right)\in\left\{1,\cdots,\frac{p-1}{2}\right\}^n$.
Define a map $\left(k_1,\cdots,k_n\right)\mapsto\left(\ell_1,\cdots,\ell_{\frac{p-1}{2}}\right)$ where $\ell_k$ is the amount of times $k$ appears in $\left(k_1,\cdots,k_n\right)$.
Clearly, $\ell_1+\cdots+\ell_{\frac{p-1}{2}}=n$ for $\left(\ell_1,\cdots,\ell_{\frac{p-1}{2}}\right)$ in the image and this assignment is invariant under the $S_n$ action on the domain.
Using this to re-index the sum above, we see that
\[
\sum_{k_1,\cdots,k_n\in\left\{1,\cdots,\frac{p-1}{2}\right\}}\prod_{m=1}^n\tau(1)\left(\zeta_{k_m}\right)c_1\left(E_m\right)=\sum_{\substack{\ell_1,\cdots,\ell_n\ge0\\\ell_1+\cdots+\ell_{\frac{p-1}{2}}=n}}\frac{n!
}{\ell_1!\cdots\ell_{\frac{p-1}{2}}!}\prod_{k=1}^{\frac{p-1}{2}}\tau(1)\left(\zeta^k\right)^{\ell_k}c_1\left(E_k\right)^{\ell_k}.
\]
Using Equations (\ref{eq: unsimplified AS class}) and (\ref{eq: big vanishing class}), we conclude
\[
\left(\prod_{k=1}^{\frac{p-1}{2}}\mc{M}^{\zeta^k}\left(E_k\right)\right)_{2n}=\sum_{k=1}^{\frac{p-1}{2}}\left(\tau(n)\left(\zeta^k\right)c_n\left(E_k\right)-\frac{1}{n!}\tau(n)\left(\zeta^k\right)c_1\left(E_k\right)^n\right).
\]
Setting this to $0$ and using Lemma \ref{lem: Ewing} shows that $c_n\left(E_k\right)=\frac{1}{n!}c_1\left(E_k\right)^n$.

For the converse, note that the computations above show the two conditions in the proposition imply $\prod_{k=1}^{\frac{p-1}{2}}\mc{M}^{\zeta^k}\left(E_k\right)=1$.
\end{proof}

The second condition of Theorem \ref{thm: Chern class restriction} can be written as $c\left(E_k\right)=e^{c_1\left(E_k\right)}$.
This motivates the following definition.

\begin{definition}\label{def: exp vector bundle}
A complex vector bundle $E$ is \emph{exponential} if its Chern classes satisfy $c_k(E)=\frac{1}{k!}c_1(E)^k$ or, equivalently, if the total Chern class is $e^{c_1(E)}$.
\end{definition}

\begin{example}\label{example: sums of line bundles}
Suppose $E=L_1\oplus\cdots\oplus L_d$ is a sum of line bundles such that $c_1(L_j)^2=0$ for each $j=1,\cdots,d$.
It follows from the additivity of the total Chern class that $c_m(E)$ is the $m$-th elementary symmetric polynomial on $c_1(L_1),\cdots,c_1(L_d)$ when $m\le d$.
The hypothesis that $c_1(L_j)^2=0$ for each $j$ implies that $c_1(E)^m=m!c_m(E)$.
This example will be generalized in Proposition \ref{prop: bundles on M} below.
\end{example}

We record the following observations.
\begin{prop}\label{prop: exp vector bundle properties}
Exponential vector bundles satisfy the following properties.
\begin{enumerate}
\item If $E_1$ and $E_2$ are exponential vector bundles then so is $E_1\oplus E_2$.
\item The pullback of an exponential vector bundle is an exponential vector bundle.
\item Exponential vector bundles have trivial Pontryagin classes.
\end{enumerate}
\end{prop}
\begin{proof}
\begin{enumerate}
\item Since $c_1(E_1\oplus E_2)=c_1(E_1)\oplus c_1(E_2)$,
\[
c(E_1\oplus E_2)=c(E_1)c(E_2)=e^{c_1(E_1)}e^{c_1(E_2)}=e^{c_1(E_1)+c_1(E_2)}=e^{c_1(E_1\oplus E_2)}.
\]
\item This follows from the naturality of Chern classes.
\item The total Pontryagin class is given by the formula
\[
p(E)=1+p_1(E)+p_2(E)+\cdots=(1+c_1(E)+c_2(E)+\cdots)(1-c_1(E)+c_2(E)-\cdots).
\]
It follows that 
\[
p_k(E)=\sum_{i+j=k}(-1)^jc_i(E)c_j(E)=\sum_{i+j=k}(-1)^j\frac{1}{i!j!}c_1(E)^k=\frac{1}{k!}\sum_{i+j=k}(-1)^j\binom{k}{j}c_1(E)^k=0.
\]
Alternatively, we may write $p(E)=e^{c_1(E)}e^{-c_1(E)}=e^0=1$.
\end{enumerate}
\end{proof}

We will only need the second and third property in Proposition \ref{prop: exp vector bundle properties}.

\begin{prop}\label{prop: euler class vanishes}
Suppose $E$ is a free $G$-vector bundle over $M$ with vanishing Atiyah-Singer class.
Then $E$ has vanishing Euler class.
\end{prop}
\begin{proof}
Let $d$ be the rank of $E$ and let $d_k$ be the rank of each eigenbundle.
By Theorem \ref{thm: Chern class restriction} the Euler class is scalar multiple of $\prod_{k=1}^{\frac{p-1}{2}}c_1\left(E_k\right)^{d_k}$ and $c_1\left(E_k\right)^{d_k+1}=0$.
Also,
\[
0=\left(\sum_{k=1}^{\frac{p-1}{2}}\tau(1)\left(\zeta^k\right)c_1\left(E_k\right)\right)^d=\prod_{k=1}^{\frac{p-1}{2}}\tau(1)\left(\zeta^k\right)^{d_k}c_1\left(E_k\right)^{d_k}
\]
where the first equality follows from Theorem \ref{thm: Chern class restriction}.
Since $\tau(1)\left(\zeta^k\right)\neq0$, this implies $\prod_{k=1}^{\frac{p-1}{2}}c_1\left(E_k\right)^{d_k}=0$ as desired.
\end{proof}

We conclude this section with a homotopical characterization of exponential vector bundles.
We will not use this result but it may be of independent interest.

\begin{prop}\label{prop: Chern character}
A complex vector bundle $E$ over $M$ is exponential if and only if its Chern character is contained in $H^0\left(M;\Q\right)\oplus H^2\left(M;\Q\right)$.
\end{prop}
\begin{proof}
Using the splitting principle, write $c_1(E)=x_1+\cdots+x_d$ where $d$ is the rank of $E$.
Generally, $c_n(E)$ is the $n$-th elementary symmetric polynomial on $x_1,\cdots,x_d$ and the Chern character is
\[
\op{ch}(E)=\sum_{j=1}^d e^{x_j}=d+\sum_{j=1}^d x_j+\frac{1}{2}\sum_{j=1}^d x_j^2+\cdots.
\]
To alleviate notation, let $e_m$ denote the $m$-th elementary symmetric polynomial on $x_1,\cdots,x_d$ and let $P_m:=\sum_{j=1}^d x_j^m$.
There are formal relations between elementary symmetric polynomials and sums of powers.
\[
P_n=\sum_{m=1}^n\left(-1\right)^{m-1}e_mP_{n-m}\hspace{1cm} e_n=\frac{1}{n}\sum_{m=1}^n\left(-1\right)^{m-1}e_{n-m}P_m
\]
Suppose $E$ is exponential.
We must show that $P_n=0$ for $n\ge2$.
By hypothesis, we have
\[
c_n(E)=\frac{1}{n!}e_1^n=e_n.
\]
We proceed by induction on $n$.
First, note that when $n=2$, the above equation becomes $\frac{1}{2}P_2+e_2=e_2$ so this case follows immediately.
Assume by induction that $P_m=0$ whenever $2\le m\le n$.
Then,
\[
P_{n+1}=\left(-1\right)^ne_{n-1}P_1+\left(-1\right)^{n+1}e_nP_0=(-1)^n\frac{e_1^n}{(n-1)!}+(-1)^{n+1}n\frac{e_1^n}{n!}=0
\]
as desired.

For the converse, suppose $P_m=0$ for $m\ge2$.
We must show $e_n=\frac{1}{n!}e_1^n$.
We proceed by induction with the case $n=1$ being vacuous.
Then,
\[
e_n=\frac{1}{n}e_{n-1}P_1=\frac{1}{n}e_{n-1}e_1=\frac{1}{n!}e_1^n
\]
which completes the proof.
\end{proof}

\section{Construction of Vector Bundles}\label{section: construction of vb}
Our goal in this section is to construct exponential vector bundles with a prescribed first Chern class $\beta$.
If we require our exponential vector bundle to have rank $N$, then $\beta$ must satisfy $\beta^{N+1}=0$.
We show in Proposition \ref{prop: bundles on M} that, up to multiplying $\beta$ by a nonzero integer, this is the only requirement.
Using these vector bundles and Theorem \ref{thm: Chern class restriction} we then prove Theorem \ref{thm: VBExistence}.

Obstruction theory will play an important role in promoting rational nullhomotopies to integral nullhomotopies.
We summarize the version of obstruction theory we need below.
We refer to \cite{Baues} for a obstruction theory when the target space is not necessarily simply connected.



\begin{theorem}\label{thm: obstruction theory nonsimply connected fiber}
Suppose $E\rightarrow B$ is a fibration with connected fiber $F$.
Suppose $(X,A)$ is a relative CW-complex and $f:X\rightarrow B$ is a map.
Let $g:X^{(n)}\cup A\rightarrow E$ be a lift of $f$ on the relative $n$-skeleton of $(X,A)$.
If $n\ge 2$, there is an obstruction class $\mc{O}(g)\in H^{n+1}\left(X,A;\pi_n F_{\rho}\right)$ where $\pi_n F_{\rho}$ denotes a local coefficient system with stalk $\pi_n F$.
This class vanishes if and only if $g$ can be redefined over the relative $n$-skeleton leaving the relative $(n-1)$-skeleton fixed so that $g$ extends as a lift to the relative $(n+1)$-skeleton.

Moreover, if $h:(Y,C)\rightarrow (X,A)$ is a cellular map then $g\circ h:Y^{(n)}\cup C\rightarrow E$ is a lift of $f\circ h:Y\rightarrow B$ defined on the relative $n$-skeleton of $(Y,C)$ and $\mc{O}(g\circ h)=h^*\mc{O}(g)\in H^{n+1}\left(Y,C;h^*\pi_nF_{\rho}\right)$.
\end{theorem}

\subsection{Obstruction Theory for $\CP^N$}
Generally, it can be difficult to show that obstructions vanish when the relevant cohomology group is nonzero.
When a space has sufficiently nice self-maps, however, pulling back cocycles allows us to find maps with vanishing obstruction cocycles.
We record some useful observations when the space $X$ is $\CP^N$.
 
An element $t\in H^2\left(\CP^N;\Z\right)$ determines a map $\CP^N\rightarrow\CP^\infty$ where the induced map on cohomology sends a generator of $H^2\left(\CP^\infty;\Z\right)$ to the element $t$.
After pushing the map into a $2N$-skeleton, we obtain a map $\lambda:\CP^N\rightarrow\CP^N$.
On $H^{2k}\left(\CP^N;\Z\right)$, $\lambda$ induces multiplication by $t^k$.
We say $\lambda$ is a \emph{scaling} map of $\CP^N$ if the corresponding integer $t$ is nonzero.

\begin{lemma}\label{lem: obstruction theory CP^N}
Suppose $E\rightarrow B$ is a fibration with connected fiber $F$.
Let $f:\CP^N\rightarrow B$ be a map whose restriction to the $2$-skeleton lifts to $E$.
If, for $n\ge2$, each $\pi_n F$ consists only of torsion elements then there is a scaling map $\lambda:\CP^N\rightarrow\CP^N$ such that $f\circ\lambda$ lifts to $E$.
\end{lemma}
\begin{proof}
Suppose there is a lift $g:\left(\CP^N\right)^{(n)}\rightarrow E$ of $f|_{\left(\CP^N\right)^{(n)}}$.
The obstruction to extending this to a lift over $\left(\CP^N\right)^{(n+1)}$ is an element of $H^{n+1}\left(\CP^N;\pi_n F\right)$ (we use that $\CP^N$ is simply connected to justify constant coefficients and we use that $\CP^N$ has only even dimensional cells to ignore the subtlety of having to redefine $g$ over the $n$-cells).

For a suitable scaling map $\lambda_n$ of $\CP^N$, $\lambda_n^*\mc{O}(g)=0$ so there is a lift of $f\circ\lambda_n$ over the $(n+1)$-skeleton of $\CP^N$.
Continuing this way shows that there is a lift of $f\circ\lambda$ for a suitable scaling map $\lambda$.
\end{proof}

\begin{lemma}\label{lem: obstruction theory CP^N nullhomotope}
Suppose $E\rightarrow B$ is a fibration with connected fiber $F$.
Let $f:\CP^N\rightarrow E$ be a map such that the composite $\CP^N\rightarrow E\rightarrow B$ is nullhomotopic.
If, for $n\ge 2$, each $\pi_n F$ consists only of torsion elements, then there is a scaling map $\lambda:\CP^N\rightarrow\CP^N$ such that $f\circ\lambda$ is nullhomotopic.
\end{lemma}
\begin{proof}
Let $C\left(\CP^N\right)$ denote the cone on $\CP^N$.
The nullhomotopy in the hypothesis gives the following diagram.
\[
\begin{tikzpicture}[scale=2]
\node (A) at (0,1) {$\CP^N$};\node (B) at (2,1) {$E$};
\node (C) at (0,0) {$C\left(\CP^N\right)$};\node (D) at (2,0) {$B$};
\path[->] (A) edge node[above]{$f$} (B) (A) edge (C) (B) edge (D) (C) edge node[above]{$g$} (D) (C) edge [dashed] (B);
\end{tikzpicture}
\]
We would like to find a map $C\left(\CP^N\right)\rightarrow E$ making the diagram commute.

Let $X_n$ denote the relative $n$-skeleton of the pair $\left(C\left(\CP^N\right),\CP^N\right)$.
Note that $X_2$ consists of only $\CP^N$, the cone point, and an edge connecting the cone point to a $\CP^N$.
By the assumption that $F$ is connected, the path in $B$ determined by the edge lifts to a path in $E$.
Hence there is a map $g_2:X_2\rightarrow E$ lifting the map $C\left(\CP^N\right)\rightarrow B$.
Let $\Sigma$ denote the suspension.
Theorem \ref{thm: obstruction theory nonsimply connected fiber} states that there is an obstruction
\[
\mc{O}\left(g_2\right)\in H^3\left(C\left(\CP^N\right),\CP^N;\pi_2 F\right)\cong H^3\left(\Sigma\left(\CP^N\right);\pi_2 F\right)\cong H^2\left(\CP^N;\pi_2 F\right)
\]
which vanishes if and only if $g_2$ can be redefined over the relative $1$-skeleton and extended to the relative $3$-skeleton.
As in the proof of Lemma \ref{lem: obstruction theory CP^N}, we can consider a scaling map $\lambda_2:\CP^N\rightarrow\CP^N$ such that the induced map on $H^2\left(\CP^N;\pi_2 F\right)$ eliminates the obstruction.
Coning $\lambda_2$ gives a map of relative CW-complexes $C\left(\lambda_2\right):\left(C\left(\CP^N\right),\CP^N\right)\rightarrow\left(C\left(\CP^N\right),\CP^N\right)$.
Since $\lambda_2\left(g_2\right)$ vanishes, we may redefine $g_2\circ C\left(\lambda_2\right):X_2\rightarrow E$ over the relative $1$-skeleton so that there is a map $g_3:X_3\rightarrow E$ lifting $g\circ C\left(\lambda_2\right):C\left(\CP^N\right)\rightarrow B$.
Continuing this way shows that, after a suitable self-map $\lambda$ of $\CP^N$, there is a lift in the diagram
\[
\begin{tikzpicture}[scale=2]
\node (A) at (0,1) {$\CP^N$}; \node (B) at (2,1) {$E$};
\node (C) at (0,0) {$C\left(\CP^N\right)$};\node (D) at (2,0) {$B$};
\path[->] (A) edge node[above]{$f\circ\lambda$} (B) (A) edge (C) (B) edge (D) (C) edge node[above]{$g\circ C\left(\lambda\right)$} (D) (C) edge (B);
\end{tikzpicture}.
\]
which proves the Lemma.
\end{proof}

\subsection{Characteristic Classes}
Recall the isomorphism $H^*(BU(N);\Z)\cong\Z[c_1,\cdots,c_N]$ where each $c_m$ has degree $2m$.
Since $c_m\in H^{2m}(BU(N);\Z)$, Brown representability identifies $c_m$ with a map $c_m:BU(N)\rightarrow K(\Z,2m)$ so the elements $c_1,\cdots,c_N$ together determine a map
\[
c_*:BU(N)\rightarrow\prod_{m=1}^{N}K(\Z,2m).
\]
By abuse of notation, we will use $c_m$ to denote both the element in $H^{2m}(BU(N);\Z)$ and the map above.
Let $x_m\in H^{2m}(K(\Z,2m);\Z)$ be a generator of the cohomology group.
Then, $c_m^*x_N=c_m$.

Rationally, there are isomorphisms
\[
H^*(K(\Z,2m);\Q)\cong\Q\left[x_m\right]
\]
and
\[
H^*\left(\prod_{m=1}^N K(\Z,2m);\Q\right)\cong \bigotimes_{m=1}^N H^*(K(\Z,2m);\Q)\cong\Q\left[x_1,\cdots x_N\right].
\]
The map $\left(c_*\right)^*$ sends $x_m$ to $c_m$ and so induces an isomorphism on rational cohomology groups.
Therefore, it induces an isomorphism on rational homotopy groups.
In particular, we have

\begin{lemma}\label{lem: fiber of c_*}
The fiber of $c_*$ is connected with torsion homotopy groups.
\end{lemma}

One can perform a similar analysis with $BSO(2N)$.
Rationally, the cohomology ring is $H^*(BSO(2N);\Q)\cong\Q\left[p_1,\cdots,p_{N-1},e\right]$ where $p_m\in H^{4m}(BSO(2N);\Q)$ are the Pontryagin classes and $e\in H^{2N}(BSO(2N);\Q)$ is the Euler class.
These classes exist integrally and so determine maps $p_*:BSO(2N)\rightarrow\prod_{m=1}^{N-1}K(\Z,4m)$ and $e:BSO(2N)\rightarrow K(\Z,2N)$.
As in the case of $BU(N)$, the map
\[
e\times p_*:BSO(2N)\rightarrow K(\Z,2N)\times\prod_{m=1}^{N-1}K(\Z,4m)
\]
induces an isomorphism of rational cohomology rings and, therefore, an isomorphism of rational homotopy groups.
This shows

\begin{lemma}\label{lem: fiber of p_*}
The fiber of $e\times p_*$ is connected with torsion homotopy groups.
\end{lemma}

Suppose $E$ is a complex vector bundle over $M$.
By abuse of notation, identify $E$ with a map $E:M\rightarrow BU(N)$.
If $\beta\in H^{2m}(M;\Z)$ is an element, then to say that $c_m(E)=\beta$ is to say that the diagram
\[
\begin{tikzpicture}[scale=2]
\node (A) at (2,1) {$BU(N)$};
\node (B) at (0,0) {$M$};\node (C) at (2,0) {$K(2m,\Z)$};
\path[->] (A) edge node[left]{$c_m$} (C) (B) edge node[above left]{$E$} (A) (B) edge node[above]{$\beta$} (C);
\end{tikzpicture}
\]
commutes.
There is a similar interpretation of the Euler and Pontryagin classes.

\subsection{Construction of Exponential Vector Bundles}
We first study the special case of exponential vector bundles on $\CP^N$.

\begin{prop}\label{prop: one exp bundle on CP^N}
Let $N'\ge N$ be an integer and let $\alpha\in H^2\left(\CP^N;\Z\right)$ be a generator.
There is an integer $t>0$ and a rank $N'$ exponential vector bundle $E$ on $\CP^N$ such that $c_1(E)=t\alpha$.
\end{prop}

\begin{proof}
First, define $\beta:=N!\alpha$, so that the classes $\frac{\beta^m}{m!}$ exist integrally.
These classes define a map
\[
\beta_*:\CP^N\rightarrow\prod_{m=1}^{N'} K(\Z,2m).
\]
We would like to find a lift in diagram
\[
\begin{tikzpicture}[scale=2]
\node (A) at (2,1) {$BU\left(N'\right)$};
\node (B) at (0,0) {$\CP^N$};\node (C) at (2,0) {$\prod_{m=1}^{N'}K(\Z,2m)$};
\path[->] (A) edge node[right]{$c_*$} (C) (B) edge node[above]{$\beta_*$} (C) (B) edge [dashed] (A);
\end{tikzpicture}
\]
where $c_*$ denotes the map determined by the Chern classes.
Such a lift need not exist but, by Lemma \ref{lem: obstruction theory CP^N}, a lift of $\beta_*\circ\lambda$ exists where $\lambda$ is a scaling map of $\CP^N$.
Let $E$ denote the vector bundle defined by this lift.
Then, $c_1\left(E\right)=t\beta$ for some nonzero integer $t$ and $c_m(E)=\frac{1}{m!}c_1(E)^m$ for $m\le N'$.
When $m>N'$ then, by our assumption that $N'\ge N$, $c_m(E)=0=\frac{1}{m!}c_1(E)^m$.
\end{proof}

Proposition \ref{prop: exp vector bundle properties} states that pullbacks of exponential vector bundles are exponential.
We obtain the following from taking further pullbacks along self-maps of $\CP^N$.

\begin{prop}\label{prop: infinitely many exp on bundles on CP^N}
Let $N'\ge N$ be an integer and let $t\alpha\in H^2\left(\CP^N;\Z\right)$ denote the class in Proposition \ref{prop: one exp bundle on CP^N}.
Then, any integer multiple of $t\alpha$ can be realized as $c_1(E)$ where $E$ is a rank $N'$ exponential vector bundle.
\end{prop}

The pullback property also allows us to construct exponential vector bundles over more general spaces.

\begin{prop}\label{prop: bundles on M}
Let $M$ be homotopy equivalent to a finite complex and let $N'\ge N$.
Then for every $\beta\in H^2\left(M;\Z\right)$ satisfying $\beta^{N+1}=0$, there is an integer $t>0$ such that, for all integers $u$, there is a rank $N'$ exponential vector bundle $E$ on $M$ with $c_1(E)=ut\beta$.
Moreover, the classifying maps $M\rightarrow BU\left(N'\right)$ associated to these bundles factor through $\CP^N$.
\end{prop}
\begin{proof}
By Proposition \ref{prop: infinitely many exp on bundles on CP^N} and the pullback property, it suffices to show that there is a map $f:M\rightarrow\CP^N$ such that $f^*\alpha$ is some nonzero integer multiple of $\beta$.

We first reduce to the case that $M$ is simply connected.
Attach $2$-cells to $M$ in order to obtain a simply connected finite complex $M'$ such that $M/M'\simeq\bigvee S^2$.
Note that $H^2\left(M';\Z\right)$ surjects onto $H^2\left(M;\Z\right)$ and $H^j\left(M';\Z\right)\cong H^j\left(M;\Z\right)$ for all $j>2$.
So, there is an element $\beta'\in H^2\left(M';\Z\right)$ mapping to $\beta$ and such that $\left(\beta'\right)^{N+1}=0$.
If the result holds for $M'$, then pulling the exponential vector bundle back along the inclusion $M\subseteq M'$ shows the result also holds for $M$.

Assuming $M$ is simply connected, there is an element in the Sullivan algebra $\left(\Lambda_{M},d_{M}\right)$ of $M$ representing $\beta$.
We will also use $\beta$ to denote this element.
The nilpotence hypothesis on $\beta$ implies there is a degree $2N+1$ element $\gamma\in\Lambda_{M}$ such that $d_{M}\gamma=\left(\beta\right)^{N+1}$.
These elements determine a map of differential graded algebras $\left(\Lambda_{\CP^N},d_{\CP^N}\right)\rightarrow\left(\Lambda_{M},d_{M}\right)$ which yields a map of rationalizations
\[
M_{(0)}\rightarrow\CP^N_{(0)}.
\]
We may assume $M$ is a finite complex so that the map $M\rightarrow M_{(0)}\rightarrow\CP^N_{(0)}$ has image in a finite subcomplex of $\CP^N_{(0)}$.
The rationalization $\CP^N_{(0)}$ can be constructed as a homotopy colimit of the diagram
\[
\CP^N\xrightarrow{\lambda^2}\CP^N\xrightarrow{\lambda^3}\CP^N\xrightarrow{\lambda^4}\CP^N\rightarrow\cdots
\]
where $\lambda^t$ is the scaling map corresponding to the integer $t$.
In particular, it is an infinite mapping telescope so the map $M\rightarrow\CP^N_{(0)}$ factors through some finite mapping telescope
\[
T_q=\varinjlim\left(\CP^N\xrightarrow{\lambda^2}\CP^N\xrightarrow{\lambda^3}\cdots\xrightarrow{\lambda^q}\CP^N\right).
\]
There is a homotopy equivalence $T_q\rightarrow\CP^N$ so we have constructed a map $M\rightarrow\CP^N$.

We now check that the pullback of $\alpha$ under this map is of the form $t\beta$.
Since $M$ is simply connected, $H^2\left(M;\Q\right)$ can be identified with $\Hom\left(\pi_2(M);\Q\right)$.
Let $\beta^*$ denote the element in $\pi_2(M)$ dual to $\beta$ under this identification.
Similarly, let $\alpha^*\in\pi_2\left(\CP^N\right)$ denote the dual to $\alpha$ and let $\alpha^*_0$ denote the image of $\alpha^*$ in $\pi_2\left(\CP^N_{(0)}\right)$.
Under the identification $\CP^N\simeq T_q$, the map $\CP^N\rightarrow \CP^N_{(0)}$ sends $\alpha^*$ to $\frac{1}{q!}\alpha^*_0$.
But the map $M\rightarrow\CP^N_{(0)}$ sends $\beta^*$ to $\alpha^*_0$.
It follows that the map $M\rightarrow\CP^N$ sends $\beta^*$ to $\frac{1}{q!}\alpha^*$.
Hence, $\alpha$ pulls back to $\frac{1}{q!}\beta$.
\end{proof}

\VBExistenceThm*

\begin{proof}[Proof of Theorem \ref{thm: VBExistence}]
Under the hypotheses of the first part, the Chern classes of each eigenbundle vanish.
Since the Chern classes determine a $BU(N)\rightarrow\prod_{m=1}^N K(\Z,2m)$ whose fiber has finite homotopy groups and $M$ is homotopy equivalent to a finite complex, there are only finitely many complex vector bundles of a fixed rank with prescribed Chern classes.

For the second part, suppose $2$ has odd order in $\left(\Z/p\Z\right)^{\times}$.
If $\beta\in H^2(M;\Z)$ is sufficiently nilpotent with respect to $V$, we may apply Proposition \ref{prop: bundles on M} to take exponential vector bundles $E_k$ such that $c_1\left(E_k\right)=u_kt\beta$ where the $u_k$ realize the linear relation $\sum_{k=1}^{\frac{p-1}{2}}u_k\tau(1)\left(\zeta^k\right)=0$.
This proves the second part of Theorem \ref{thm: VBExistence}.
\end{proof}

In order to address Question \ref{question: realizing exotic actions}, we would like a converse to Theorem \ref{thm: VBExistence}.
The difficulty in obtaining a converse is number theoretic; we know that any subset of $\left\{\tau(1)\left(\zeta^k\right)\right\}$ of size $u+1$ has a $\Q$-linear relation but it is not clear whether there exist smaller subsets which are $\Q$-linearly dependent.
However, we can say the following.

\begin{prop}\label{prop: manifolds with infinitely many bundles}
Suppose $M$ is homotopy equivalent to a finite complex and that there are infinitely many $G$-vector bundles with vanishing Atiyah-Singer class.
Then, infinitely many of these $G$-vector bundles are pulled back from $G$-vector bundles over $\CP^N$.
\end{prop}
\begin{proof}
We may assume that there is a vector bundle $E$ with vanishing Atiyah-Singer class and such that some of the $c_1\left(E_k\right)$ are nonzero rationally.

Let $\beta\in H^2(M;\Q)$ be one of the nonzero $c_1\left(E_k\right)$ where the $N$ such that $c_1\left(E_k\right)^{N+1}=0$ is minimal.
By projecting the relation $\sum_{k=1}^{\frac{p-1}{2}}\tau(1)\left(\zeta^k\right)c_1\left(E_k\right)=0$ to the $\Q(\zeta)$-subspace of $H^2\left(M;\Q(\zeta)\right)$ spanned by $\beta$, we see that there is a linear relation $\sum_{k=1}^{\frac{p-1}{2}}\tau(1)\left(\zeta^k\right)\beta_k$ where $\beta_k$ is a rational multiple of $\beta$ and $\beta_k=0$ if $c_1\left(E_k\right)=0$.
Moreover, our choice of $\beta$ ensures that the dimension of the eigenspace $V_k$ is at least $N$ when $\beta_k\neq0$.
By scaling the $\beta_k$ simultaneously, we may use Proposition \ref{prop: bundles on M} to realize $\beta_k$ as the first Chern class of an exponential bundle factoring through $\CP^N$.
Adding these together gives a bundle over $M$ with vanishing Atiyah-Singer class which factors through $\CP^N$.
Composing with self-maps of $\CP^N$ gives infinitely many such vector bundles.
\end{proof}

\section{Block Bundles}\label{section: block bundles}
We recall some definitions and facts about block bundles.
We refer to \cite{CassonThesis} and \cite{RourkeSandersonDelta2} for a more detailed treatment.

\begin{definition}\label{def: block bundle}
Let $K$ be a finite simplicial complex and let $Y$ be a polyhedron.
Let $\pi:E\rightarrow\abs{K}$ be a continuous map.
A \emph{block chart} for a simplex $\sigma\subseteq K$ is a $PL$-homeomorphism
\[
h_\sigma:\pi^{-1}(\sigma)\rightarrow\sigma\times Y
\]
such that, for each face $\tau\le\sigma$, the restriction $h_\sigma|_{\pi^{-1}(\tau)}$ is a $PL$-homeomorphism $\pi^{-1}(\tau)\rightarrow\tau\times Y$.
We say that $\pi:E\rightarrow\abs{K}$ is a \emph{block bundle with fiber $Y$} if there is a block chart for every simplex $\sigma\subseteq K$.
\end{definition}

It is not true that, for a $PL$-block bundle, $\pi^{-1}(x)\cong F$ for an arbitrary point $x\in\abs{K}$; this distinguishes block bundles from fiber bundles.
One can define block bundles with other structure groups.
We will only be concerned with $PL$-block bundles.
Our block bundles will typically be over smooth manifolds in which case we give the manifold a $PL$-structure compatible with the smoothing.

\begin{definition}\label{def: iso and equiv of block bundles}
Let $\pi_i:E_i\rightarrow\abs{K}$ be $PL$-block bundles for $i=0,1$.
An \emph{isomorphism} of $PL$-block bundles is a $PL$-homeomorphism $H:E_0\rightarrow E_1$ such that $H\left(\pi_0^{-1}(\sigma)\right)=\pi_1^{-1}(\sigma)$ for all simplices $\sigma\subseteq K$.

The block bundles $\pi_0$ and $\pi_1$ are \emph{equivalent} if there is a subdivision $K'$ of $K$ such that $\pi_0$ and $\pi_1$ determine isomorphic block bundles over $K'$.
\end{definition}

Casson shows in \cite{CassonThesis} that equivalence of $PL$-block bundles is an equivalence relation.

\begin{definition}\label{def: SPL(Y)}
Let $Y$ be a polyhedron.
Define $\widetilde{PL}(Y)$ to be the simplicial group whose $d$-simplices are the $PL$-homeomorphisms $f:\Delta^d\times Y\rightarrow\Delta^d\times Y$ such that, for each face $\sigma\subseteq\Delta^d$,
\[
f\left(\pi_{\Delta^d}^{-1}(\sigma)\right)\subseteq\pi_{\Delta^d}^{-1}(\sigma)
\]
where $\pi_{\Delta^d}:\Delta^d\times Y\rightarrow\Delta^d$ is the projection.
If $Y$ is an orientable $PL$-manifold, define $\widetilde{SPL}(Y)$ to be the simplicial group whose $d$-simplices are the orientation preserving $PL$-homeomorphisms $f:\Delta^d\times Y\rightarrow\Delta^d\times Y$ satisfying the above property.
\end{definition}

One can construct classifying spaces $B\widetilde{PL}(Y)$ for $PL$-block bundles with fiber $Y$.
The following is \cite[Theorem 2]{CassonThesis}.

\begin{theorem}\label{thm: classifying equiv of block bundles}
There is a bijection between equivalence classes of $PL$-block bundles over $K$ with fiber $Y$ and homotopy classes of maps $[K,B\widetilde{PL}(Y)]$.
\end{theorem}

Finally, we record a consequence of the fact that block bundles are controlled over the base space.
Let $SO^G(V)$ denote the group of $G$-equivariant orientation preserving linear transformations of $V$.
Suppose $E_0$ and $E_1$ are two $G$-vector bundles such that the composites $M\rightarrow BSO^G(V)\rightarrow B\widetilde{SPL}(SV/G)$ are homotopic.
Let $D_0$ and $D_1$ be the respective unit disk bundles and let $SE_0$ and $SE_1$ denote the respective unit sphere bundles.
The homotopy $M\times I\rightarrow B\widetilde{SPL}(SV/G)$ gives a concordance $W$ of $PL$-block bundles $SE_1/G$ and $SE_2/G$.
Let $\widetilde{W}$ denote the $G$-cover.
We may form the $G$-manifold $E':=\widetilde{W}\cup_{SE_1}D_1$ which has boundary $SE_0$.

\begin{prop}\label{prop: block bundle control}
In the situation above, suppose $f:\partial E_1/G\rightarrow SE_0/G$ is an equivalence of $PL$-block bundles over $M$.
Let $\tilde{f}$ denote the map on covers.
Then there is an equivariant homeomorphism $E'\rightarrow D_0$ restricting to $\tilde{f}$ on the boundary. 
\end{prop}
\begin{proof}
First, note that $D_0\setminus M$ is equivariantly homeomorphic to $SE_0\times[0,\infty)$.
We construct an equivariant homeomorphism $W\cup_{SE_1}D_1\setminus M\rightarrow SE_0\times[0,\infty)$ such that the restriction to the boundary is $\tilde{f}$ and we show that this homeomorphism extends to $M$.

By hypothesis, $W$, $SE_0/G$ and $SE_1/G$ have the same classifying map.
So after taking a subdivision of $M\times I$, there is an isomorphism of $PL$-block bundles $F:W/G\rightarrow SE_0/G\times[0,1]$ which restricts to $f$ on $SE_0/G\times\{0\}$.
Let $M_0$ denote the triangulation of $M\times\{0\}$ and let $M_1$ denote the triangulation of $M\times\{1\}$.
Let $f_1$ denote the isomorphism $F|_{SE_1/G}:SE_1/G\rightarrow SE_0/G$.

Write $W_j$ for the trivial $PL$-block bundle over $M\times[j,j+1]$ where we equip $M$ with a triangulation subordinate to the barycentric subdivision of $M_j$.
Subdivide $W_j$ so that there is an isomorphism of $PL$-block bundles $F_j:W_j\rightarrow M\times[j,j+1]\times SE_0/G$ restricting to $f_j$ on the part over $M\times\{j\}$.
Define $M_{j+1}$ to be the triangulation on $M\times\{j+1\}$ and define $f_{j+1}$ to be the restriction of $F_j$ to the part over $M\times\{j+1\}$.

Continuing this way, we obtain a homeomorphism
\[
W/G\cup_{SE_1/G}\left(SE_1/G\times[1,\infty)\right)\rightarrow SE_0/G\times[0,\infty).
\]
Lifting to the $G$-cover gives an equivariant homeomorphism
\[
W\cup_{SE_1}\left(SE_1\times[1,\infty)\right)\rightarrow SE_0\times[0,\infty).
\]
This extends continuously to $M$; any sequence in $W\cup_{SE_1}\left(SE_1\times[1,\infty)\right)$ approaching a point $m\in M$ will get sent to a sequence on the right hand side approaching the same point.
\end{proof}

\subsection{The Rational Homotopy Type of $B\widetilde{SPL}(SV/G)$}

In \cite{CappellWeinbergerSimpleAS}, Cappell-Weinberger describe $B\widetilde{SPL}(SV/G)$ rationally.
Let $\tilde{L}^s_k(G)$ denote the reduced simple $L$-space of $G$; this is a space satisfying $\pi_n\tilde{L}^s_k(G)=\tilde{L}^s_{n+k}(G)$ where the right hand side denotes the reduced simple $L$-groups.

\begin{theorem}[Cappell-Weinberger]\label{thm: rational homotopy of SPL(Lens)}
There is a map
\[
B\widetilde{SPL}(SV/G)\rightarrow B\widetilde{SPL}(SV)\times\tilde{L}^s_{\dim_{\R} V}(G)_{(0)}.
\]
whose fiber is connected with torsion homotopy groups.
\end{theorem}

\begin{remark}
Cappell-Weinberger state that the map in Theorem \ref{thm: rational homotopy of SPL(Lens)} is a $\frac{1}{2\abs{G}}$-equivalence.
They do not show that $B\widetilde{SPL}(SV/G)$ is simply connected.
Their proof shows that $\pi_1 B\widetilde{SPL}(SV/G)$ is a finite solvable group with a composition series having $\abs{G}$-torsion abelian subquotients.
They also show that the map on higher homotopy groups is an equivalence after inverting $2\abs{G}$.
\end{remark}

In Theorem \ref{thm: rational homotopy of SPL(Lens)}, the map $B\widetilde{SPL}(SV/G)\rightarrow B\widetilde{SPL}(SV)$ is given by pulling back a homeomorphism of $\Delta^d\times SV/G$ to $\Delta^d\times SV$.
In particular, if an equivariant vector bundle is non-equivariantly trivial, then the composite
\[
M\rightarrow BSO^G(V)\rightarrow B\widetilde{SPL}(SV/G)\rightarrow B\widetilde{SPL}(SV)
\]
is nullhomotopic.

The other component of the map involves the Atiyah-Singer class and an argument with the Conner-Floyd isomorphism (a more detailed treatment of an analogous argument may be found in \cite[Chapter 4]{MadsenMilgram}).
First recall the rational equivalence $\tilde{L}^s_{\dim_{\R} V}(G)_{(0)}\simeq BO\left(\widetilde{RO}(G)\right)_{(0)}\times\Omega^2BO\left(\widetilde{RO}(G)\right)_{(0)}$ where $BO\left(\widetilde{RO}(G)\right)$ denotes $\Omega^\infty$ of $KO$ smashed with Moore spectrum.
To define a map $B\widetilde{SPL}(SV/G)\rightarrow BO\left(\widetilde{RO}(G)\right)_{(0)}$ it suffices to define an element of $KO^0\left(B\widetilde{SPL}(SV/G);\widetilde{RO}(G)_{(0)}\right)$.

The universal coefficients theorem for $KO$ gives an isomorphism
\[
KO^0\left(X;\widetilde{RO}(G)_{(0)}\right)\rightarrow \Hom\left(KO_0(X),\widetilde{RO}(G)_{(0)}\right)
\]
for finite complexes $X$.
An inverse limit argument shows that, for infinite $X$, there is a surjection
\[
KO^0\left(X;\widetilde{RO}(G)_{(0)}\right)\rightarrow\varprojlim\Hom\left(KO_0\left(X^{(i)}\right),\widetilde{RO}(G)_{(0)}\right)
\]
where the limit on the right is taken over skeleta.

The Conner-Floyd isomorphism states that $\Omega^{SO}_{4*+i}\left(X\right)\otimes_{\Omega^{SO}_*(*)}\Z\left[\frac{1}{2}\right]\cong KO_i\left(X;\Z\left[\frac{1}{2}\right]\right)$.
So given an element of
\[
\Hom\left(\Omega^{SO}_{4*}\left(X\right)\otimes_{\Omega^{SO}_*(*)}\Z\left[\frac{1}{2}\right],\widetilde{RO}(G)_{(0)}\right)
\]
we obtain an element of $KO^0\left(X;\widetilde{RO}(G)_{(0)}\right)$ and hence a map
\[
X\rightarrow BO\left(\widetilde{RO}\left(G\right)\right)_{(0)}.
\]
This map is unique up to homotopy if $X$ is a finite complex.

We now define the homomorphism $AS\in\Hom\left(\Omega^{SO}_{4*}\left(X\right)\otimes_{\Omega^{SO}_*(*)}\Z\left[\frac{1}{2}\right],\widetilde{RO}(G)_{(0)}\right)$ giving rise to the map $B\widetilde{SPL}(SV/G)\rightarrow BO\left(\widetilde{RO}\left(G\right)\right)_{(0)}$.
Suppose $f:M\rightarrow B\widetilde{SPL}(SV/G)$ represents an element of $\Omega^{SO}\left(B\widetilde{SPL}(SV/G)\right)$.
Let $SE/G\rightarrow M$ be the corresponding block bundle.
Then $SE/G$ has a $G$-cover $SE$ which is a block bundle over $M$ with fiber $SV$.
Since $G$ acts freely on $SE$ and because $SE$ bounds non-equivariantly, there is an integer $r>0$ such that $r$-many copies of $SE$ bounds a manifold $X$ on which $G$ acts freely.
Define
\[
AS([f]):=\frac{1}{r}\op{sign}_G\left(X\right)-\op{sign}\left(E\right)\cdot\op{triv}
\]
where $\op{sign}_G$ denotes the $\widetilde{RO}(G)$-valued multisignature, $\op{sign}\left(E\right)$ denotes the (non-equivariant) signature of the block bundle obtained by coning the sphere bundle and $\op{triv}$ denotes the trivial representation.

So far, only ``half'' of the map $B\widetilde{SPL}(SV/G)\rightarrow L^s_{\dim_{\R}V}(G)_{(0)}$ has been defined; we still need to define a map
\[
B\widetilde{SPL}(SV/G)\rightarrow\Omega^2BO\left(\widetilde{RO}(G)\right)_{(0)}.
\]
This is equivalent to a map $\Sigma^2B\widetilde{SPL}(SV/G)\rightarrow BO\left(\widetilde{RO}(G)\right)_{(0)}$.
As above, we obtain such a map from a group homomorphism
\[
\Omega_{4*}\left(\Sigma^2B\widetilde{SPL}(SV/G)\right)\otimes_{\Omega^{SO}_*(*)}\Z\left[\frac{1}{2}\right]\cong\Omega_{4*+2}\left(B\widetilde{SPL}(SV/G)\right)\otimes_{\Omega^{SO}_*(*)}\Z\left[\frac{1}{2}\right]\xrightarrow{AS}\widetilde{RO}(G)_{(0)}.
\]
This homomorphism is defined using the Atiyah-Singer invariant in an identical manner.

Suppose the $f:M\rightarrow B\widetilde{SPL}(SV/G)$ factors through $BSO^G(V)$.
Then we may regard $SE$ above as the sphere bundle of the corresponding $G$-vector bundle and, by \cite[Section 7]{AtiyahSinger3}, $AS([f])$ is the element of $\widetilde{RO}(G)$ with character
\[
AS([f])(g)=\left\langle A(g,V) L(M)\mc{M}(g,E),[M]\right\rangle.
\]
The following lemma asserts that the character is determined by its value on a generator of $G$.

\begin{lemma}\label{lem: AS determined by generator}
Suppose $\xi=\bigoplus_{k=1}^{\frac{p-1}{2}}\xi_k$ and $E=\bigoplus_{k=1}^{\frac{p-1}{2}} E_k$ are two $G$-vector bundles with fiber $V$ and with eigenbundle decompositions associated to a generator $g_0$ of $G$.
If
\[
\prod_{k=1}^{\frac{p-1}{2}}\mc{M}^{\zeta^k}\left(\xi_k\right)=\prod_{k=1}^{\frac{p-1}{2}}\mc{M}^{\zeta^k}\left(E_k\right)
\]
then $AS(M,\xi)=AS(M,E)$.
\end{lemma}
\begin{proof}
For an arbitrary $g_0^n\in G$, we have
\[
AS(M,E)\left(g_0^n\right)=\left\langle A\left(g_0^n,V\right)\mc{L}(M)\prod_{k=1}^{\frac{p-1}{2}}\mc{M}^{\zeta^{nk}}\left(E_k\right),[M]\right\rangle.
\]
Let $\sigma\in Gal(\Q(\zeta)/\Q)$ be the automorphism defined by $\zeta\mapsto\zeta^n$.
By Lemma \ref{lem: Galois invariance of as},
\begin{align*}
AS(M,E)\left(g_0^n\right)&=\left\langle A\left(g_0^n,V\right)\mc{L}(M)\sigma\left(\prod_{k=1}^{\frac{p-1}{2}}\mc{M}^{\zeta^k}\left(E_k\right)\right),[M]\right\rangle\\
&=\left\langle A\left(g_0^n,V\right)\mc{L}(M)\sigma\left(\prod_{k=1}^{\frac{p-1}{2}}\mc{M}^{\zeta^k}\left(\xi_k\right)\right),[M]\right\rangle\\
&=AS(M,\xi)\left(g_0^n\right).
\end{align*}
\end{proof}

\begin{prop}\label{prop: BSPL(SV/G) and AS invariant}
Suppose $M$ is homotopy equivalent to a finite complex and $\nu_1,\nu_2:M\rightarrow BSO^G(V)$ determine $G$-vector bundles with equal Atiyah-Singer classes.
Then the compositions
\[
M\rightarrow BSO^G(V)\rightarrow B\widetilde{SPL}(SV/G)\rightarrow\tilde{L}^s_{\dim_{\R} V}(G)_{(0)}
\]
is are homotopic.
\end{prop}
\begin{proof}
Since $M$ is homotopy equivalent to a finite complex, if the two induced maps
\[
\Hom\left(\Omega_{4*}^{SO}\left(B\widetilde{SPL}(SV/G)\right)\otimes_{\Omega_*^{SO}(*)}\Z\left[\frac{1}{2}\right],\widetilde{RO}(G)\right)\rightarrow \Hom\left(\Omega_{4*}^{SO}(M)\otimes_{\Omega_*^{SO}(*)}\Z\left[\frac{1}{2}\right],\widetilde{RO}(G)\right)
\]
send the map $AS$ to the same element then they determine the same map $M\rightarrow BO\left(\widetilde{RO}(G)\right)_{(0)}$.
For $i=0,1$, let $\tilde{\nu}_i$ denote the composition
\[
M\xrightarrow{\nu_i}BSO^G(V)\rightarrow B\widetilde{SPL}(SV/G).
\]
Suppose $h:M'\rightarrow M$ represents an element of $\Omega_{4*}^{SO}(M)$.
Then $\nu_0\circ h$ and $\nu_1\circ h$ classify $G$-vector bundles over $M'$ with equal Atiyah-Singer classes.
So the Atiyah-Singer invariants of the two $SV/G$-block bundles over $M'$ are equal.
It follows that $\tilde{\nu}_i\circ h$ represent elements of $\Omega_{4*}^{SO}\left(B\widetilde{SPL}(SV/G)\right)$ such that $AS\left([\tilde{\nu}_0\circ h]\right)=AS\left([\tilde{\nu}_1\circ h]\right)$.

A similar argument shows that $\nu_1$ and $\nu_2$ determine the same map $M\rightarrow\Omega^2 BO\left(\widetilde{RO}(G)\right)_{(0)}$.
\end{proof}

Using Proposition \ref{prop: BSPL(SV/G) and AS invariant} we can show that vector bundles with vanishing Atiyah-Singer class can be taken to have lens space bundles which are trivial as $PL$-block bundles.

\begin{prop}\label{prop: exotic normal bundle}
Suppose $M$ is homotopy equivalent to a finite complex.
Suppose there are infinitely many $G$-vector bundles over $M$ with fiber $V$ and vanishing Atiyah-Singer class.
Then, infinitely many of these $G$-vector bundles $E$ satisfy the following properties.
\begin{enumerate}
\item The classifying map of $E$ factors through $\CP^N$,
\item The corresponding lens space bundle $SE/G$ is isomorphic as a $PL$-block bundle to the trivial bundle $M\times SV/G$.
\end{enumerate}
\end{prop}

\begin{proof}
As before, we begin with the case $M=\CP^N$.
By hypothesis and Proposition \ref{prop: BSPL(SV/G) and AS invariant}, the composite
\[
\CP^N\xrightarrow{E}BSO^G(V)\rightarrow B\widetilde{SPL}(SV/G)\rightarrow\tilde{L}^s_{\dim_{\R}V}(G)_{(0)}
\]
is nullhomotopic.

We show that the composite
\[
\CP^N\xrightarrow{E}BSO^G(V)\rightarrow B\widetilde{SPL}(SV/G)\rightarrow B\widetilde{SPL}(SV)
\]
also vanishes.
There is a commuting diagram
\[
\begin{tikzpicture}[scale=2]
\node (A) at (0,1) {$\CP^N$};\node (B) at (2,1) {$BSO^G(V)$};\node (C) at (4,1) {$B\widetilde{SPL}(SV/G)$};
\node (D) at (2,0) {$BSO(V)$};\node (E) at (4,0) {$B\widetilde{SPL}(SV)$};
\path[->] (A) edge node[above]{$E$} (B) (B) edge (C) (B) edge (D) (C) edge (E) (D) edge (E);
\end{tikzpicture}
\]
where the left vertical map is obtained by forgetting the action.
So it suffices to show that $E$ is trivial as a (non-equivariant) real vector bundle.
The Pontryagin classes of $E$ vanish by Proposition \ref{prop: exp vector bundle properties} and the Euler class vanishes by Proposition \ref{prop: euler class vanishes}.
We have shown that the composition
\[
\CP^N\xrightarrow{E} BSO^G(V)\rightarrow BSO(V)\xrightarrow{e\times p_*}K\left(\Z,\dim_{\R}V\right)\times\prod_{m=1}^{\dim_{\R}V/2-1}K(\Z,4m)
\]
is nullhomotopic.
Applying Lemma \ref{lem: obstruction theory CP^N} to the fibration $BSO(2d)\xrightarrow{e\times p_*}K\left(\Z,2d\right)\times\prod_{m=1}^{d-1}K(\Z,4m)$ shows that, after composing with a scaling map $\lambda$ of $\CP^N$, the composite $E\circ\lambda:\CP^N\rightarrow BSO^G(V)\rightarrow BSO(2d)$ is nullhomotopic.

Replace $E$ with the $\lambda^*E$ so that it is trivial as a (non-equivariant) real vector bundle.
So far, we have constructed an $E=\bigoplus_{k=1}^{\frac{p-1}{2}} E_k$ such that some $c_1\left(E_k\right)$ is a nonzero and such that the composite
\[
\CP^N\xrightarrow{E}BSO^G(V)\rightarrow B\widetilde{SPL}(SV/G)\rightarrow B\widetilde{SPL}(SV)\times\tilde{L}^s_{2d}(G)_{(0)}
\]
is nullhomotopic.
It remains to show that we can modify $E$ so that the map is nullhomotopic integrally.
We apply Lemma \ref{lem: obstruction theory CP^N nullhomotope} to the fibration $B\widetilde{SPL}(SV/G)\rightarrow B\widetilde{SPL}(SV)\times \tilde{L}^{s}_{\dim_{\R}V}(G)_{(0)}$ to see that, for a scaling map $\lambda:\CP^N\rightarrow\CP^N$,
\[
\CP^N\xrightarrow{\lambda}\CP^N\xrightarrow{E}BSO^G(V)\rightarrow B\widetilde{SPL}(SV/G)
\]
is nullhomotopic.

For the general case, apply Proposition \ref{prop: manifolds with infinitely many bundles} to the result on $\CP^N$.
\end{proof}

\section{Smoothing $PL$-Concordances}\label{section: smoothing}
In this section, we would like to take advantage of the fact that a closed manifold of dimension at least $5$ has only finitely many smooth structures in order to show that the vector bundles constructed above, whose lens space bundles have homeomorphic total spaces, can be made to have diffeomorphic total spaces.

Let us first recall some classical facts about smoothing.
We refer the reader to \cite{HirschMazur} for details.

\begin{definition}\label{def: smoothing}
Given a closed $PL$-manifold $M$, a \emph{smoothing} of $M$ (or a \emph{smooth structure} on $M$) is a smooth manifold $M_0$ and a $PL$-homeomorphism $f_0:M_0\rightarrow M$.
Two smooth structures $f_i:M_i\rightarrow M$, $i=0,1$, are \emph{concordant} if there is a smooth structure on the $PL$-manifold $M\times I$ and a $PL$-homeomorphism $F:M\times I\rightarrow M\times I$ restricting to $f_i$ on $M\times\{i\}$.
\end{definition}

Let $PL/O(M)$ denote the concordance classes of smoothings of $M$.
This set can be identified with concordance classes of linear structures on the stable tangent $PL$-microbundle.
Specifying an initial smooth structure $\eta$ on $M$ gives a bijection $PL/O(M)\cong[M,PL/O]$ and $PL/O$ has an $H$-space structure induced by Whitney sum.
Let $PL/O(M,\eta)$ denote the set of concordance classes of smoothings with a specified smooth structure $\eta$.
Thus $PL/O(M,\eta)$ is a homotopy functor from the category of smooth manifolds and continuous maps to abelian groups.

\subsection{Differentiable Vector Bundles}
Suppose $p:E\rightarrow M$ is a vector bundle where $M$ is smoothable and let $\mc{U}$ be an atlas on $M$ such that $p$ is locally trivial over each $U\in\mc{U}$.
Given a smooth structure $\eta$ on $M$, we may assume that the transition maps $U_\alpha\cap U_\beta\rightarrow GL_k(\R)$ are smooth for $U_\alpha$ and $U_\beta$ in some subatlas $\mc{U}'$.
The vector bundle equipped with a maximal subatlas satisfying this property is called a \emph{differentiable vector bundle}.
By \cite[p.89, Theorem 1.9]{HirschMazur}, every vector bundle over a smooth manifold admits the structure of a differentiable vector bundle and this structure is unique.

The total space of a differentiable vector bundle has a unique smooth structure such that the local trivializations $p^{-1}U\rightarrow U\times\R^n$ are smooth.
It turns out that this assignment yields a well-defined bijection $p^!:PL/O(M)\rightarrow PL/O(E)$ \cite[p. 93, Theorem 2.6]{HirschMazur}.
If we wish to consider these as pointed sets, then there is a well-defined bijection $p^!:PL/O(M,\eta)\rightarrow PL/O\left(E,p^!\eta\right)$.
It is important to note that there are generally multiple ways of making the total space $E$ a vector bundle over $M$ and different ways of doing so result in different bijections.

Using the structure group $SO^G(V)$, can define differentiable $G$-vector bundles over a smooth manifold (with trivial $G$-action) similarly and the proof of \cite[p.89, Theorem 1.9]{HirschMazur} shows that every $G$-vector bundle admits a unique differentiable $G$-vector bundle structure.
When $V$ is a free representation, this gives the corresponding lens space bundle a smooth structure.

\subsection{Functoriality}
Suppose $f:(M,\eta)\rightarrow (N,\omega)$ is a continuous map between smooth manifolds.
Hirsch-Mazur \cite[p. 111]{HirschMazur} give the following description of the induced map $PL/O(f):PL/O(N,\omega)\rightarrow PL/O(M,\eta)$.
Let $\phii:(N,\beta)\rightarrow(N,\omega)$ represent a smooth structure in $PL/O(N,\omega)$ which we will denote $[\beta]$.
Let $\rho$ denote the standard smooth structure on $\R$.
For some sufficiently large integer $d$, there is a smooth embedding
\[
\psi:(M,\eta)\rightarrow \left(N\times\R^d,\omega\times\rho^d\right)
\]
such that $\pi_N\circ\psi$ is homotopic to $f$.
Then the normal bundle $\nu$ of $\psi(M)\subseteq\left(N\times\R^d,\omega\times\rho^d\right)$ determines a vector bundle on the $PL$-submanifold $\psi(M)\subseteq \left(N\times\R^d,\beta\times\rho^d\right)$.
The total space $E(\nu)$ has a smooth structure $\nu^!\eta$ coming from the smooth structure $\eta$ on $M$ and the vector bundle structure.
This space can also be identified with an open subset of $\left(N\times\R^d,\beta\times\rho^d\right)$ hence it inherits a smooth structure $\left[\beta\times\rho^d\right]\in\mc{S}^{PL/O}\left(E(\nu),\nu^!\eta\right)$.
Since $\nu^!:PL/O(M,\eta)\rightarrow PL/O\left(E(\nu),\nu^!\eta\right)$ is a bijection, there is a unique smooth structure $[\alpha]\in PL/O(M,\eta)$ such that $\nu^!([\alpha])=\left[\beta\times\rho^d\right]$.
The smoothing $\alpha$ represents $PL/O(f)\left(\left[\beta\right]\right)$.

The next result essentially states that given a smooth map between manifolds, the smooth structure on a pullback $PL$-block bundle is the pullback smooth structure induced by maps of total spaces.

\begin{prop}\label{prop: smoothing pullback}
Suppose $F, M$ and $N$ are smooth manifolds.
Let $E_0$ and $E_1$ be $F$-bundles over $N$ such that $E_0$ and $E_1$ are smooth and the projections to $N$ are smooth.
Let $\phii:E_1\rightarrow E_0$ be an isomorphism of $PL$-block bundles and let $f:M\rightarrow N$ be smooth.
Let $\eta$ denote the given smooth structure on $E_0$ and let $f^*\eta$ denote the smooth structure on $f^*E_0$ making it a smooth submanifold of $M\times E_0$ and let $f^*:PL/O(E_0,\eta)\rightarrow PL/O(f^*E_0,f^*\eta)$ be the induced map on smooth structures.
Then, $f^*[\phii]$ is represented by the induced map on pullbacks $f^*E_1\rightarrow f^*E_0$.
\end{prop}

\begin{proof}
Let $\psi:M\rightarrow\R^d$ be a smooth embedding.
This determines a smooth embedding $f^*E_0\rightarrow E_0\times\R^d$ which sends $x\in f^*E_0$ to $(f(x),\psi\circ \pi_{f^*E_0}x)$ where $\pi_{f^*E_0}:f^*E_0\rightarrow M$ is the bundle projection.
Let $\nu_0$ denote the normal bundle of this embedding.
The total space $E\left(\nu_0\right)$ inherits a smooth structure as an open subset of $E_0\times\R^d$.
Let us call this structure $\gamma$.
Also, there is the smooth structure $\nu_0^!f^*\eta$ coming from the vector bundle structure.
By Hirsch-Mazur's description of the induced map,
\[
\nu_0^!\left(PL/O(f)\left(\left[SE/G\right]\right)\right)=[\gamma]
\]
in the set $PL/O\left(E\left(\nu_0\right),\nu_0^!f^*\eta\right)$.

Let $W$ denote a concordance of $PL$-block bundles between $E_0$ and $E_1$.
Then, $f^*W$ is a concordance of $PL$-block bundles between $f^*E_0$ and $f^*E_1$.
Moreover, there is an isomorphism $F:W\rightarrow E_1\times I$ of $PL$-block bundles over $N\times I$.
Let $\pi_M$ denote the composite $f^*W\rightarrow M\times I\rightarrow M$ and consider the $PL$-embedding
\[
\left(F\times\op{id}_{\R^d}\right)\circ\left(f\times\left(\psi\circ\pi_M\right)\right):f^*W\rightarrow W\times\R^d\rightarrow E_1\times I\times\R^d.
\]
Over $0\in I$, this restricts to the embedding $f^*E_0 \rightarrow E_0\times\R^d$ above and over $1\in I$, this restricts to a smooth embedding $f^*E_1\rightarrow E_1\times\R^d$.
Let $\nu_1$ denote the normal bundle of the second embedding.
By taking an open neighborhood of the image of $F\left(f\times\left(\psi\circ\pi_M\right)\right)$, we see that the smooth structure $[\gamma]$ above is the same as $\nu_1^!\left[f^*\phii\right]$.
So it suffices to show that $\nu_0$ and $\nu_1$ are isomorphic as vector bundles.

Let $\nu$ denote the normal bundle of the smooth embedding $f\times\psi:M\rightarrow N\times\R^d$.
The pullback of $\nu$ to $W$ restricts to $\nu_0$ over $0\in I$ and $\nu_1$ over $1\in I$ which shows that $\nu_0$ and $\nu_1$ are isomorphic vector bundles.
\end{proof}

\subsection{Smooth Trivialization of $SE/G$}
We now show that the isomorphism $SE/G\rightarrow M\times SV/G$ of lens space block bundles over $M$ can be made into a diffeomorphism of total spaces.
As before, the main tool will be the use of scaling maps of $\CP^N$.

\begin{prop}\label{prop: smoothing SE/G on CP^N}
Suppose $E$ is a $G$-vector bundle over $\CP^N$ and let $f:SE/G\rightarrow\CP^N\times SV/G$ be an equivalence of $PL$-block bundles.
Then, there is a scaling map $\lambda:\CP^N\rightarrow\CP^N$ such that $\lambda^*f:\lambda^* SE/G\rightarrow\lambda^*\left(\CP^N\times SV/G\right)$ is $PL$-isotopic to a diffeomorphism. 
\end{prop}
\begin{proof}
Identify $\left[\CP^N\times SV/G,PL/O\right]$ with the smoothings of $\CP^N\times SV/G$ by specifying the product smooth structure.
The isomorphism of $PL$-block bundles $f:SE/G\rightarrow \CP^N\times SV/G$ determines an element $[f]\in\left[\CP^N\times SV/G,PL/O\right]$.
Since $PL/O$ is an infinite loop space, there is a generalized cohomology theory $\mathbf{E}^*$ such that $\mathbf{E}^0(X)=[X,PL/O]$.
Explicitly, if $PL/O=\Omega^n \mathbf{E}_n$, then for $n\ge0$, $\mathbf{E}^n(X)=\left[X,\mathbf{E}_n\right]$ and $\mathbf{E}^{-n}(X)=\left[X,\Omega^n PL/O\right]$.
In particular, there is an Atiyah-Hirzebruch-Serre spectral sequence
\[
H^i\left(\CP^N;\mathbf{E}^j(SV/G)\right)\Rightarrow \mathbf{E}^{i+j}\left(\CP^N\times SV/G\right).
\]
Let $X_n$ denote $\pi_{\CP^N}^{-1}\left(\left(\CP^N\right)^{(n)}\right)$, the preimage of the $n$-skeleton of $\CP^N$ under the projection.
Convergence means that there is a filtration
\[
\cdots F_n\subseteq F_{n-1}\subseteq\cdots\subseteq F_0= \mathbf{E}^{i+j}\left(\CP^N\times SV/G\right)
\]
where, for $n>0$, $F_n$ is the kernel of the restriction
\[
\mathbf{E}^{i+j}\left(\CP^N\times SV/G\right)\rightarrow \mathbf{E}^{i+j}\left(X_{n-1}\right)
\]
such that the $E_\infty$-terms of the spectral sequence are subquotients of the filtration.
We will only be interested in the case where $i+j=0$ in which case $E_{\infty}^{n,-n}=F_n/F_{n+1}$.
We may assume that over a vertex $x_0$ of $\CP^N$, $f$ restricts to a diffeomorphism so $[f]$ vanishes under the restriction
\[
\left[\CP^N\times SV/G,PL/O\right]\xrightarrow{x_0^*}\left[SV/G,PL/O\right].
\]
In particular, $[f]\in F_1$.

Now, if $\lambda:\CP^N\rightarrow\CP^N$ is a scaling map, it induces a map of fiber bundles
\[
\lambda:\lambda^*\left(\CP^N\times SV/G\right)\cong\CP^N\times SV/G\rightarrow \CP^N\times SV/G
\]
and hence a morphism of spectral sequences.
For $n>1$, this induces multiplication by $t^n$ on $H^n\left(\CP^N;\mathbf{E}^{-n}(SV/G)\right)$ for some integer $t$.
Since the homotopy groups of $PL/O$ are finite, so are the groups $\left[SV/G,\Omega^n PL/O\right]=\mathbf{E}^{-n}(SV/G)$.
It follows that, by choosing an appropriate $\lambda$, $F_1$ is in the kernel of the induced map
\[
\lambda^*:\mathbf{E}^0\left(\CP^N\times SV/G\right)\rightarrow \mathbf{E}^0\left(\CP^N\times SV/G\right).
\]
In particular, $[f]\in\left[\CP^N\times SV/G,PL/O\right]$ vanishes after pulling back along $\lambda$.
The result now follows from Proposition \ref{prop: smoothing pullback}.
\end{proof}

Using the case for $\CP^N$, we can give an analogous statement for bundles over $M$.

\begin{prop}\label{prop: smoothing SE/G on M}
Suppose a $G$-vector bundle over a smooth manifold $M$ given by Proposition \ref{prop: exotic normal bundle} is classified by the composite
\[
M\xrightarrow{\beta}\CP^N\xrightarrow{E'} BSO^G(V).
\]
Then, there is a scaling map $\lambda$ of $\CP^N$ such that the $G$-vector bundle $E$ classified by
\[
M\xrightarrow{\beta}\CP^N\xrightarrow{\lambda}\CP^N\xrightarrow{E'} BSO^G(V)
\]
gives a lens space bundle $SE/G$ which is equivalent as a $PL$-block bundle to $M\times SV/G$.
Moreover this equivalence is $PL$-isotopic to a diffeomorphism.
\end{prop}
\begin{proof}
Proposition \ref{prop: smoothing SE/G on CP^N} shows that there is a positive degree self-map $\lambda$ of $\CP^N$ such that $\lambda^*E'/G$ is equivalent as a $PL$-block bundle to $\CP^N\times SV/G$ and that this equivalence is $PL$-isotopic to a diffeomorphism.
Applying Proposition \ref{prop: smoothing pullback} to $\beta^*:\left[\CP^N\times SV/G,PL/O\right]\rightarrow\left[M\times SV/G,PL/O\right]$ gives the desired result.
\end{proof}

\subsection{Proof of Theorem \ref{thm: FPRealization}}
We can now prove
\FPRealization*

\begin{proof}
Let $X$ and $M\subseteq X^G$ be as in the theorem.
Define $\bar{X}$ to be the complement of an equivariant tubular neighborhood of $M$ in $X$.
Then $\bar{X}$ has a $G$ action and $\partial\bar{X}$ is equivariantly diffeomorphic to $M\times SV$.

By Proposition \ref{prop: exotic normal bundle} and Proposition \ref{prop: smoothing SE/G on M}, for infinitely many of these vector bundles $E$, there are isomorphisms of $PL$-block bundles $SE/G\rightarrow M\times SV/G$ which are $PL$-concordant to diffeomorphisms.
Let $f:SE/G\rightarrow M\times SV/G$ denote the diffeomorphism and let $\tilde{f}$ denote its lift on $G$-covers.
Define the smooth $G$-manifold $Y:=\bar{X}\cup_{\tilde{f}}E$.

It remains to construct an equivariant homeomorphism $g:Y\rightarrow X$.
On $\bar{X}$, we take $g$ to be the identity so we just need to construct an equivariant homeomorphism $g:DE\rightarrow M\times DV$ where $D$ denotes the unit disk bundle and such that $g$ restricts to $\tilde{f}$ on the boundary.
Since $f$ is $PL$-concordant to an equivalence of $PL$-block bundles, there is a $PL$-isomorphism $F:SE/G\times I\rightarrow M\times SV/G\times I$ such that $F|_{SE/G\times\{0\}}=f$ and $F|_{SE/G\times\{1\}}$ is an equivalence of $PL$-block bundles over $M$.
Writing $DE=SE\times I\cup DE$ and defining $g$ to be the lift of $F$ on $SE\times I$, we may assume instead that $f$ is an equivalence of $PL$-block bundles over $M$.
Proposition \ref{prop: block bundle control} shows that $\tilde{f}$ may be extended to an equivariant homeomorphism.
This proves the first part.

For the second part, just note that by taking scaling maps of $\CP^N$, the first Chern classes are being multiplied by constants $t$ with $\abs{t}>1$.
\end{proof}

\section{Nontrivial Normal Bundles}\label{section: nontrivial normal bundles}
So far, we have concentrated on the case where the normal bundle of $M$ is trivial as a $G$-vector bundle.
If this assumption is removed, the characteristic class computations become much more difficult and there is not much we are able to say.
Suppose $\xi=\bigoplus_{k=1}^{\frac{p-1}{2}}\xi_k$ and $E=\bigoplus_{k=1}^{\frac{p-1}{2}} E_k$ are $G$-vector bundles over $M$.
In order for the Atiyah-Singer classes to be equal, one sees that, in cohomological degree $2$,
\[
\sum\tau(1)\left(\zeta^k\right)\left(c_1\left(\xi_k\right)-c_1\left(E_k\right)\right)=0
\]
so the classes $c_1\left(\xi_k\right)-c_1\left(E_k\right)$ must realize the linear relation between the $\tau(1)\left(\zeta^k\right)$.
One can also derive a condition for $c_2$.

\begin{prop}\label{prop: c_2 of nontrivial vector bundle}
If $\xi$ and $E$ have the same Atiyah-Singer class, then
\[
c_2\left(\xi_k\right)-c_2\left(E_k\right)=\frac{1}{2}c_1\left(\xi_k\right)^2-\frac{1}{2}c_1\left(E_k\right)^2.
\]
\end{prop}
\begin{proof}
In cohomological degree $4$, we have
\begin{align*}
0&=\left(\prod_{k=1}^{\frac{p-1}{2}}\mc{M}^{\zeta^k}(\xi_k)\right)_{4}-\left(\prod_{k=1}^{\frac{p-1}{2}}\mc{M}^{\zeta^k}(E_k)\right)_{4}\\
=&\sum_{k=1}^{\frac{p-1}{2}}\tau(2)(\zeta^k)(c_2(\xi_k)-c_2(E_k))+\sum_{k=1}^{\frac{p-1}{2}}\tau(1,1)(\zeta^k)(c_1(\xi_k)^2-c_1(E_k)^2)
\\&+\sum_{k_1\neq k_2}\tau(1)(\zeta_{k_1})\tau(1)(\zeta_{k_2})(c_1(\xi_{k_1})c_1(\xi_{k_2})-c_1(E_{k_1})c_1(E_{k_2})).
\end{align*}

We use the conditions on the first Chern class to simplify the expression.
\begin{align*}
0&=\left(\sum_{k=1}^{\frac{p-1}{2}}\tau(1)(\zeta^k)(c_1(\xi_k)-c_1(E_k))\right)^2\\
&=\sum_{k_1,k_2=1}^{\frac{p-1}{2}}\tau(1)(\zeta_{k_1})\tau(1)(\zeta_{k_2})\left(c_1(\xi_{k_1})c_1(\xi_{k_2})-c_1(E_{k_1})c_1(\xi_{k_2})-c_1(\xi_{k_1})c_1(E_{k_2})+c_1(E_{k_1})c_1(E_{k_2})\right)
\end{align*}
We can square
\[
2\sum_{k=1}^{\frac{p-1}{2}}\tau(1)\left(\zeta^k\right)c_1\left(\xi_k\right)=\sum_{k=1}^{\frac{p-1}{2}}\tau(1)\left(\zeta^k\right)\left(c_1\left(\xi_k\right)+c_1\left(E_k\right)\right)
\]
to obtain
\begin{align*}
&4\sum_{k_1,k_2=1}^{\frac{p-1}{2}}\tau(1)\left(\zeta_{k_1}\right)\tau(1)\left(\zeta_{k_2}\right)c_1\left(\xi_{k_1}\right)c_1\left(\xi_{k_2}\right)\\
&=\sum_{k_1,k_2=1}^{\frac{p-1}{2}}\tau(1)\left(\zeta_{k_1}\right)\tau(1)\left(\zeta_{k_2}\right)\left(c_1\left(\xi_{k_1}\right)c_1\left(\xi_{k_2}\right)+c_1\left(\xi_{k_1}\right)c_1\left(E_{k_2}\right)+c_1\left(E_{k_1}\right)c_1\left(\xi_{k_2}\right)+c_1\left(E_{k_1}\right)c_1\left(E_{k_2}\right)\right).
\end{align*}
Now, subtracting $2\sum_{k_1,k_2=1}^{\frac{p-1}{2}}\tau(1)\left(\zeta_{k_1}\right)\tau(1)\left(\zeta_{k_2}\right)\left(c_1\left(\xi_{k_1}\right)c_1\left(\xi_{k_2}\right)+c_1\left(E_{k_1}\right)c_1\left(E_{k_2}\right)\right)$ from both sides gives
\begin{align*}
&2\sum_{k_1,k_2=1}^{\frac{p-1}{2}}\tau(1)\left(\zeta_{k_1}\right)\tau(1)\left(\zeta_{k_2}\right)\left(c_1\left(\xi_{k_1}\right)c_1\left(\xi_{k_2}\right)-c_1\left(E_{k_1}\right)c_1\left(E_{k_2}\right)\right)\\
&=\sum_{k_1,k_2}^{\frac{p-1}{2}}\tau(1)\left(\zeta_{k_1}\right)\tau(1)\left(\zeta_{k_2}\right)\left(-c_1\left(\xi_{k_1}\right)c_1\left(\xi_{k_2}\right)+c_1\left(\xi_{k_1}\right)c_1\left(E_{k_2}\right)+c_1\left(E_{k_1}\right)c_1\left(\xi_{k_2}\right)-c_1\left(E_{k_1}\right)c_1\left(E_{k_2}\right)\right).
\end{align*}
Our previous computation shows that this is $0$.

We use this to cancel out a lot of the classes showing up in the Atiyah-Singer formula.
We obtain
\begin{align*}
0&=\sum_{k=1}^{\frac{p-1}{2}}\tau(2)\left(\zeta^k\right)\left(c_2\left(\xi_k\right)-c_2\left(E_2\right)\right)-\tau(1)\left(\zeta^k\right)^2\left(c_1\left(\xi_k\right)^2-c_1\left(E_k\right)^2\right)\\
&+\tau(1,1)\left(\zeta^k\right)\left(c_1\left(\xi_k\right)^2-c_1\left(E_k\right)^2\right)\\
&=\sum_{k=1}^{\frac{p-1}{2}}\tau(2)\left(\zeta^k\right)\left(c_2\left(\xi_k\right)-c_2\left(E_k\right)\right)-\frac{1}{2}\left(c_1\left(\xi_k\right)^2-c_1\left(E_k\right)^2\right).
\end{align*}
By linear independence of $\{\tau(2)(\zeta^k)\}$, we obtain $c_2\left(\xi_k\right)-c_2\left(E_k\right)=\frac{1}{2}\left(c_1\left(\xi_k\right)^2-c_1\left(E_k\right)^2\right)$.
\end{proof}

This condition on $c_2$ shows that our analysis of the trivial bundle case does not na\"ively extend to the nontrivial case.

\begin{example}
Let $M=\CP^2\#\overline{\CP^2}$.
The cohomology ring of $M$ is $H^*\left(M;\Z\right)\cong\Z[a,b]/\left(a^3=b^3=0,ab=0,a^2=-b^2\right)$.
Now, let $\xi=\bigoplus_{k=1}^{\frac{p-1}{2}}\xi_k$ where each eigenbundle $\xi_k$ is a line bundle with $c_1\left(\xi_k\right)=a$.
Suppose $E=\bigoplus_{k=1}^{\frac{p-1}{2}} E_k$ has the same Atiyah-Singer class.
Then each $E_k$ must be a line bundle so Proposition \ref{prop: c_2 of nontrivial vector bundle} gives
\[
c_1\left(E_k\right)^2=c_1\left(\xi_k\right)^2=a^2.
\]
Writing $c_1\left(E_k\right)=xa+yb$, this equation becomes
\[
\left(x^2-y^2\right)a^2=a^2.
\]
Since $x^2-y^2=1$ has only finitely many integer solutions, we conclude that there only finitely many cohomology classes appear as $c_1\left(E_k\right)$.
Finally, note that $(a+b)^2=0$ so there is a nonzero element sufficiently nilpotent with respect to the representation.
\end{example}

Proposition \ref{prop: c_2 of nontrivial vector bundle} suggests a condition such as
\[
c_m\left(\xi_k\right)-c_m\left(E_k\right)=\frac{1}{m!}c_1\left(\xi_k\right)^m-\frac{1}{m!}c_1\left(E_k\right)^m
\]
is the correct way of generalizing the exponential condition.
However, we have not been able to modify the proof of \ref{thm: Chern class restriction} to this more general case.
In order to show that sums of exponential vector bundles give rise to trivial lens space block bundles, we also had to show that the Pontryagin classes and the Euler class vanish.
It follows from Proposition \ref{prop: c_2 of nontrivial vector bundle} that $p_1\left(\xi_k\right)=p_1\left(E_k\right)$ when $\xi$ and $E$ have the same Atiyah-Singer class.
In general, there is no reason to expect the Pontryagin classes or the Euler class of $\xi$ and $E$ to be equal when $\xi$ and $E$ have the same Atiyah-Singer class.

Our methods do give smoothings when the normal bundle has a large trivial summand.
By factoring a normal bundle $M\rightarrow BSO^G(V)$ through $Y\times\CP^N$ where $Y$ is a product of Grassmannians, we can prove the following.

\begin{theorem}\label{thm: nontrivial normal bundles}
Let $G=\Z/p\Z$ where $p$ is such that $2$ has odd order in $\left(\Z/p\Z\right)^{\times}$.
Let $X$ be a smooth $G$-manifold and let $M$ be a component of $X^G$.

Suppose the normal bundle $\nu$ of $M$ is of the form $\nu_0\oplus \ep_{V_1}$ where $\ep_{V_1}$ denotes the trivial $G$-vector bundle for some representation $V_1$.
If there is a nonzero element $\beta\in H^2\left(M;\Q\right)$ sufficiently nilpotent with respect to $V_1$, then $\overline{TOP/O}_G\left(X\right)$ is infinite.
\end{theorem}
\begin{proof}
Let $E$ be a $G$-vector bundle over $M$ with fiber $V$ whose lens space bundle is trivial as a $PL$-block bundle constructed as in Theorem \ref{thm: VBExistence}.
Since $SO^G\left(V_0\right)$ is a product of unitary groups, there is a product of Grassmannian manifolds $Y$ such that the classifying map for $\nu\oplus E$ factors as follows.
\[
\begin{tikzpicture}[scale=2]
\node (A) at (0,1) {$M$};\node (B) at (1,1) {$Y\times\CP^N$};\node (C) at (3,1) {$BSO^G\left(V_0\right)\times BSO^G\left(V_1\right)$};\node (D) at (6,1) {$BSO^G(V_0\oplus V_1)$};
\node (E) at (3,0) {$B\widetilde{SPL}\left(SV_0/G\right)\times B\widetilde{SPL}\left(SV_1/G\right)$};\node (F) at (6,0) {$B\widetilde{SPL}\left(S\left(V_0\oplus V_1\right)/G\right)$};
\path[->] (A) edge (B) (B) edge node[above]{$\gamma\times E'$} (C) (C) edge (D) (C) edge (E) (E) edge (F) (D) edge (F);
\end{tikzpicture}
\]
The map $M\rightarrow\CP^N$ is determined by a nonzero multiple of the cohomology class $\beta$.
By abuse of notation, we will denote this map by $\beta$.
The map $E':\CP^N\rightarrow BSO^G\left(V_1\right)$ defines a $G$-vector bundle whose lens space bundle is trivial as a $PL$-block bundle.
We may assume that the $\zeta^k$-eigenbundle $E'_k$ of $E'$ (corresponding to the eigenbundle decomposition with respect to some fixed generator $g_0\in\Z/p\Z$) has nonzero first Chern class.

We first consider the bundle $\gamma\times E'$ over $M=Y\times\CP^N$.
Since the lens space bundle of $E'$ gives a trivial $PL$-block bundle, we see that the lens space bundle of the $G$-vector bundle $\gamma\times E'$ is isomorphic as a $PL$-block bundle to $\gamma\times \ep_{V_1}$.
This gives us an element of $PL/O(S(\gamma\times\ep_{V_1})/G,\eta)$ where $\eta$ is the smooth structure on the lens space bundle given by considering $\gamma\times\ep_{V_1}$ as a differentiable $G$-vector bundle.
Note that $S\left(\gamma\times\ep_{V_1}\right)/G$ is a bundle over $\CP^N$ with fiber the total space of $S\left(\gamma\times\ep_{V_1}|_{Y\times\{*\}}\right)/G$.
In particular, we may apply the Atiyah-Hirzebruch-Serre spectral sequence
\[
H^i\left(\CP^N;\mathbf{E}^{-i}\left(S\left(\gamma\times\ep_{V_1}|_{Y\times\{*\}}\right)/G\right)\right)\Rightarrow\left[S\left(\gamma\times\ep_{V_1}\right)/G,PL/O\right]
\]
and argue as in Proposition \ref{prop: smoothing SE/G on CP^N} to see that, for some scaling map $\lambda$ of $\CP^N$, there is a $PL$-block bundle isomorphism $\lambda^*S\left(\gamma\times E'\right)/G\rightarrow S\left(\gamma\times\ep_{V_1}\right)/G$ which is $PL$-isotopic to a diffeomorphism.
Moreover, $\lambda^*\left(\gamma\times E'\right)=\gamma\times\lambda^*E'$ so $c_1\left(\lambda^*\left(\gamma\times E'\right)\right)=c_1\left(\gamma\right)+tc_1\left(E'\right)$ for some nonzero $t$.
This equation also holds on eigenbundles so $c_1\left(\lambda^*\left(\gamma\times E'\right)\right)\neq c_1\left(\left(\gamma\times\ep_{V_1}\right)_k\right)$.

As a consequence, $\nu_0\oplus\beta^*\lambda^*E'$ is a $G$-vector bundle over $M$ such that $c_1\left(\left(\nu_0\oplus\beta^*\lambda^*E'\right)_k\right)-c_1\left(\left(\nu_0\oplus\ep_{V_1}\right)_k\right)$ is a nonzero integer multiple of $\beta$.
Also, there is an isomorphism of $PL$-block bundles $S\left(\nu_0\oplus\beta^*\lambda^*E'\right)/G\rightarrow S\left(\nu_0\oplus\ep_{V_1}\right)/G$ which is $PL$-isotopic to a $PL$-diffeomorphism.
Proceeding as in the proof of Theorem \ref{thm: FPRealization}, we see that $\nu_0\oplus\beta^*\lambda^*E'$ is an exotic normal bundle of $(X,M)$.
Finally, taking further compositions with scaling maps $Y\times\CP^N\rightarrow Y\times\CP^N$ shows that there are infinitely many such bundles whose $\zeta^k$-eigenbundles have distinct first Chern classes.
\end{proof}

\subsection{Stable Smoothings}\label{subsection: Stability}
\begin{definition}\label{def: stable smoothing}
Let $X$ be a $G$-manifold.
A \emph{stable $G$-smoothing} of $X$ is a $G$-smoothing $\alpha:Y\rightarrow X\times\rho$ where $\rho$ is some $G$-representation.
Two stable $G$-smoothings $\alpha_i:Y_i\rightarrow X\times\rho_i$, $i=0,1$ are \emph{stably isotopic} if there are representations $\sigma_i$ such that $\rho_0\oplus\sigma_0\cong\rho_1\oplus\sigma_1$ and the smooth $G$-structures $\alpha_i\times\op{id}_{\sigma_i}:Y_i\times\sigma_i\rightarrow X\times\rho_i\times\sigma_i$ are $G$-isotopic.
Let $TOP/O_G^{st}\left(X\right)$ denote the stable isotopy classes of stable $G$-smoothings.
\end{definition}

\StabilityThm*

\begin{proof}
By Theorem \ref{thm: nontrivial normal bundles}, after taking the product with a sufficiently large representation $\rho$, the set $TOP/O_G\left(X\times\rho\right)$ is infinite.
For a smoothing $\alpha:Y\rightarrow X\times\rho$, let $\left(\nu_{\alpha^{-1}}M\right)_k$ denote the $\zeta^k$-eigenbundle of the normal bundle of $\alpha^{-1}(M)$.
The construction of these smoothings yields an infinite set $\left\{\alpha_j:Y_j\rightarrow X\times\rho\right\}$ such that $c_1\left(\left(\nu_{\alpha_i^{-1}M}\right)_k\right)\neq c_1\left(\left(\nu_{\alpha_j^{-1}M}\right)_k\right)$ for some $k$ and whenever $i\neq j$.
If $\sigma$ is a $G$-representation,
\[
c_1\left(\left(\nu_{\alpha_i^{-1}M}\right)_k\right)=c_1\left(\left(\nu_{\left(\alpha_i\times\op{id}_{\sigma}\right)^{-1}M}\right)_k\right)
\]
so we see that
\[
c_1\left(\left(\nu_{\left(\alpha_i\times\op{id}_\sigma\right)M}\right)_k\right)\neq c_1\left(\left(\nu_{\left(\alpha_j\times\op{id}_\sigma\right)M}\right)_k\right)
\]
whenever $i\neq j$.
Therefore, the smoothings constructed in Theorem \ref{thm: nontrivial normal bundles} are not stably isotopic.
\end{proof}

\bibliographystyle{alpha}
\bibliography{ChernClassObstructions}

\begin{thebibliography}{MM79}

\bibitem[AS68]{AtiyahSinger3}
M.~F. Atiyah and I.~M. Singer.
\newblock The index of elliptic operators. {III}.
\newblock {\em Ann. of Math. (2)}, 87:546--604, 1968.

\bibitem[Bau77]{Baues}
Hans~J. Baues.
\newblock {\em Obstruction theory on homotopy classification of maps}.
\newblock Lecture Notes in Mathematics, Vol. 628. Springer-Verlag, Berlin-New
  York, 1977.

\bibitem[BH78]{BrowderHsiangProblem}
W.~Browder and W.~C. Hsiang.
\newblock Some problems on homotopy theory manifolds and transformation groups.
\newblock In {\em Algebraic and geometric topology ({P}roc. {S}ympos. {P}ure
  {M}ath., {S}tanford {U}niv., {S}tanford, {C}alif., 1976), {P}art 2}, volume
  XXXII of {\em Proc. Sympos. Pure Math.}, pages 251--267. Amer. Math. Soc.,
  Providence, RI, 1978.

\bibitem[Cas96]{CassonThesis}
A.~J. Casson.
\newblock Generalisations and applications of block bundles.
\newblock In {\em The {H}auptvermutung book}, volume~1 of {\em $K$-Monogr.
  Math.}, pages 33--67. Kluwer Acad. Publ., Dordrecht, 1996.

\bibitem[CF64]{ConnerFloyd}
P.~E. Conner and E.~E. Floyd.
\newblock {\em Differentiable periodic maps}, volume Band 33 of {\em Ergebnisse
  der Mathematik und ihrer Grenzgebiete, (N.F.)}.
\newblock Academic Press, Inc., Publishers, New York; Springer-Verlag,
  Berlin-G\"{o}ttingen-Heidelberg, 1964.

\bibitem[CW91]{CappellWeinbergerSimpleAS}
Sylvain Cappell and Shmuel Weinberger.
\newblock A simple construction of {A}tiyah-{S}inger classes and piecewise
  linear transformation groups.
\newblock {\em J. Differential Geom.}, 33(3):731--742, 1991.

\bibitem[Ewi76]{EwingSpheresAsFPSets}
John Ewing.
\newblock Spheres as fixed point sets.
\newblock {\em Quart. J. Math. Oxford Ser. (2)}, 27(108):445--455, 1976.

\bibitem[Ewi78]{EwingSemifree}
John Ewing.
\newblock Semifree actions on finite groups on homotopy spheres.
\newblock {\em Trans. Amer. Math. Soc.}, 245:431--442, 1978.

\bibitem[Hir66]{Hirzebruch}
F.~Hirzebruch.
\newblock {\em Topological methods in algebraic geometry}, volume Band 131 of
  {\em Die Grundlehren der mathematischen Wissenschaften}.
\newblock Springer-Verlag New York, Inc., New York, german edition, 1966.

\bibitem[HM74]{HirschMazur}
Morris~W. Hirsch and Barry Mazur.
\newblock {\em Smoothings of piecewise linear manifolds}.
\newblock Annals of Mathematics Studies, No. 80. Princeton University Press,
  Princeton, N. J.; University of Tokyo Press, Tokyo, 1974.

\bibitem[KM63]{KervaireMilnor}
Michel~A. Kervaire and John~W. Milnor.
\newblock Groups of homotopy spheres. {I}.
\newblock {\em Ann. of Math. (2)}, 77:504--537, 1963.

\bibitem[KS77]{KirbySiebenmann}
Robion~C. Kirby and Laurence~C. Siebenmann.
\newblock {\em Foundational essays on topological manifolds, smoothings, and
  triangulations}.
\newblock Annals of Mathematics Studies, No. 88. Princeton University Press,
  Princeton, N.J.; University of Tokyo Press, Tokyo, 1977.
\newblock With notes by John Milnor and Michael Atiyah.

\bibitem[Las79]{LashofStableGSmoothing}
Richard Lashof.
\newblock Stable {$G$}-smoothing.
\newblock In {\em Algebraic topology, {W}aterloo, 1978 ({P}roc. {C}onf.,
  {U}niv. {W}aterloo, {W}aterloo, {O}nt., 1978)}, volume 741 of {\em Lecture
  Notes in Math}, pages pp 283--306. Springer, Berlin, 1979.

\bibitem[MM79]{MadsenMilgram}
Ib~Madsen and R.~James Milgram.
\newblock {\em The classifying spaces for surgery and cobordism of manifolds}.
\newblock Annals of Mathematics Studies, No. 92. Princeton University Press,
  Princeton, N.J.; University of Tokyo Press, Tokyo, 1979.

\bibitem[RS71]{RourkeSandersonDelta2}
C.~P. Rourke and B.~J. Sanderson.
\newblock {$\triangle $}-sets. {II}. {B}lock bundles and block fibrations.
\newblock {\em Quart. J. Math. Oxford Ser. (2)}, 22:465--485, 1971.

\bibitem[Sch79]{SchultzSpherelike}
Reinhard Schultz.
\newblock Spherelike {$G$}-manifolds with exotic equivariant tangent bundles.
\newblock In {\em Studies in algebraic topology}, volume~5 of {\em Adv. in
  Math. Suppl. Stud.}, pages 1--38. Academic Press, New York-London, 1979.

\end{thebibliography}
\end{document}